\documentclass[reqno]{amsart}
\usepackage[margin=3cm]{geometry}
\usepackage{amsmath}
\usepackage{amssymb}
\usepackage{amsbsy}
\usepackage[latin1]{inputenc}
\usepackage{dsfont}
\usepackage{color}

\usepackage{empheq}
\usepackage{pifont} 
\usepackage{shadethm}
\usepackage{subfigure,graphicx}
\usepackage{lineno}
\usepackage[table]{xcolor}
\usepackage{makecell}
\usepackage{float}
\usepackage{multirow}

\def\dd {{\boldsymbol d}}
\def\f {{\boldsymbol f}}

\def\u {{\boldsymbol u}}
\def\n {{\boldsymbol n}}
\def\x {{\boldsymbol x}}
\def\vv {{\boldsymbol vv}}
\def\w {{\boldsymbol w}}

\def\Bphi {{\boldsymbol \phi}}

\def\utecho{{\widehat{\boldsymbol u}}}
\def\Om {{\Omega}}

\def\DD{{\boldsymbol D}}

\def\HH{{\boldsymbol H}}
\def\VV{{\boldsymbol V}}
\def\LL{{\boldsymbol L}}
\def\WW{{\boldsymbol W}}
\def\XX{{\boldsymbol  X}}
\def\YY{{\boldsymbol  Y}}

\def\Bmat{{\boldsymbol{\mathsf{B}}}}
\def\Emat{{\boldsymbol{\mathsf{E}}}}
\def\Mmat{{\boldsymbol{\mathsf{M}}}}

\def\Lmat{{\boldsymbol{\mathsf{L}}}}

\def\Wvec{\mathsf{W}}
\def\Dvec{\mathsf{D}}

\def\Fvec{\mathsf{F}}

\def\R {{\rm I}\hskip -0.85mm{\rm R}}

\newtheorem{theorem}{Theorem}%[section]
\newtheorem{lemma}[theorem]{Lemma}
\newtheorem{proposition}[theorem]{Proposition}

\newtheorem{remark}[theorem]{Remark}

\begin{document}

\title[Time-splitting approximation of nematic liquid crystal flows]{A projection-based time-splitting algorithm for approximating nematic liquid crystal flows with stretching}

\author{R.C. Cabrales}
\thanks{Departamento de Ciencias B\'asicas,
Universidad del B\'\i o-B\'\i o, Casilla 447, Chill\'an, Chile.
E-mail: {\tt roberto.cabrales@gmail.com}. Partially supported under Chilean grant
fondecyt 1140074.} 
\author{F. Guill\'en-Gonz\'alez}\thanks{Dpto.~E.D.A.N. and IMUS, Universidad 
de Sevilla, Aptdo.~1160, 41080 Sevilla, Spain. 
E-mail: {\tt guillen@us.es}. Partially supported by  Ministerio de Econom\'ia y Competitividad under Spanish grant MTM2015-69875-P with the participation of FEDER}
\author{J. V. Guti\'errez-Santacreu}\thanks{ Dpto. de 
Matem\'atica Aplicada I, Universidad de Sevilla, E.
T. S. I. Inform\'atica. Avda. Reina Mercedes, s/n. 41012 Sevilla,
Spain. {\tt juanvi@us.es}. Partially supported by  Ministerio de Econom\'ia y Competitividad  under Spanish grant MTM2015-69875-P with the participation of FEDER}

\maketitle

\begin{abstract} 
 A numerical method is developed for solving a system of partial differential equations modeling the flow of a nematic liquid crystal fluid with stretching effect, which takes into account the geometrical shape of its molecules. 
This system couples the velocity vector, the scalar pressure and the 
director vector representing the direction along which the molecules 
are oriented. The scheme is designed by using finite elements in space 
and a time-splitting algorithm to uncouple the calculation of the 
variables: the velocity and pressure are computed by using a 
projection-based algorithm and the director is computed jointly to 
an auxiliary variable. Moreover, the computation of this auxiliary 
variable can be avoided at the discrete level by using piecewise constant 
finite elements in its approximation. Finally, we use a pressure 
stabilization technique  allowing a stable equal-order interpolation 
for the velocity and the pressure. Numerical experiments concerning 
annihilation of singularities are presented to show the stability 
and efficiency of the scheme.

\end{abstract}

\noindent{\bf Mathematics Subject Classification:} Nematic liquid crystal; Finite elements; Projection method; Time-splitting method.

\noindent{\bf Keywords:} 35Q35, 65M60, 76A15

%\section{Ideas para el articulo}
%\begin{enumerate}
%\item Una posibilidad para la estructura del articulos es: iniciar con la introduccion,
%donde se presenten las ecuaciones y el significado de cada termino 
%(ver \cite{Cavaterra2012}),
%se comenten cosas que se han hecho numericamente (quizas teoricamente tambien), con
%dos parrafos finales donde se explique el aporte del trabajo y el resumen de las proximas secciones.
%Luego de la introduccion, podemos pasar a la seccion FE approximation, donde se coloquen
%los preliminares de espacios de Sobolev (aunque esto podria perfectamente ir al final de la 
%introduccion), los de elementos finitos (incluyendo hipotesis)  y la descripcion del metodos.
%Luego pasamos a una seccion donde estaran los estimaciones a priori y finalmente la seccion
%de los experimentos numericos.
%\item Los experimentos a presentarse, es el de aniquilacion de dos y cuatro singularidades.
%Hay dos tablas con el estudio de la estabilidad, falta ver si incluyo comentarios sobre
%lo que pasa cuando refinamos h y k (inclusive $\varepsilon$).
%\end{enumerate}

\section {Introduction}
For a long time we have all believed that matter only existed in three states: 
solid, liquid, and gas. However, liquid crystals are substances that combine 
features of both isotropic liquids and crystalline solids, exhibiting intermediate
 transitions between solid and liquid phases, called \emph{mesophases}. This behavior 
 is due, in part, to the fact that liquid crystals are made up of macromolecules 
 of similar size. Moreover it is well-known that the shape of the molecules plays 
 an important role in the flow regimes of liquid crystal fluids. Thus, in 
 general, different behaviors are expected for the dynamics of liquid crystal 
 fluids being constituted by molecules of different shapes. For instance, liquid
  crystal fluids built from disk-shaped molecules may exhibit strikingly 
  different properties from those composed of rod-shape materials. 

The mathematical theory for the hydrodynamics of liquid crystal fluids was 
initiated by Ericksen \cite{Ericksen1961, Ericksen1987}, and Leslie 
\cite{Leslie1979, Leslie1968}. Such a theory describes liquid crystals 
according to the different degrees of positional and orientational ordering of 
their molecules. The former refers to the relative position, on average, of 
the molecules or groups of molecules, while the latter alludes to the fact 
that the molecules tend to be locally aligned toward a specific and preferred
direction, described by a unit vector, called \emph{director}, defined according to 
the form of the molecules.

Within the liquid crystal phases, we find the so-called \emph{nematic} phase. In it, 
the molecules have no positional order, but they self-organize to have 
long-range directional order. Thus, the molecules flow freely and their center 
of mass positions are randomly distributed as in a liquid, while maintain their 
long-range directional order. The breakdown of the self-alignment causes the
 appearance of defects or singularities which are able to significantly influence 
 the flow behavior.  

In this paper we are interested in the numerical solution of the flow of a 
nematic liquid crystal governed by an Ericksen-Leslie-type system which incorporates  
the stretching effect. This stretching effect comes from the kinematic transport 
of the director field, which depends on the shape of the molecules.    

Let  $T>0$ be a fixed time and let $\Omega\subset\mathbb{R}^M, M=2$ or $3$, be 
a bounded open set with boundary $\partial\Omega$. Set $Q=\Omega\times(0,T]$ 
and $\Sigma=\partial\Omega\times(0,T]$. Then the equations are written as 
follows (see \cite{Lin2007}, \cite{Liu2007} for more physical background, derivation, and discussion):
\begin{equation}\label{LC}
\left\{
\begin{array}{rcl}
         \partial_t \dd+  (\u\cdot \nabla) \dd +\beta (\nabla\u)\dd
         +(1+\beta)(\nabla\u)^T\dd+\gamma ( \f(\varepsilon,\dd) -\delta\dd ) 
         & =& \boldsymbol{0}\quad\mbox{ in  $Q$,}%\label{LC-eq1}
         \\
        \partial_t\u + (\u\cdot \nabla) \u - \nu \delta \u+\nabla p 
        + \lambda\nabla\cdot ((\nabla \dd)^T \nabla \dd)&&
        \\
        +\lambda \nabla\cdot (\beta (\f(\varepsilon,\dd)-\delta\dd)\dd^T
        +(1+\beta) \dd(\f(\varepsilon,\dd)-\delta\dd)^T)
        & =& \boldsymbol{0}\quad\mbox{ in $ Q$,}%\label{LC-eq2}
        \\
        \nabla\cdot \,\u & = & 0\quad\mbox{ in $ Q$.}%\label{LC-eq3}
\end{array}
\right.
\end{equation}
We complete this system with homogeneous Dirichlet conditions for the 
velocity field  (non-slip) 
and homogeneous Neumann  boundary conditions  for the director field:
\begin{equation}\label{boundary-conditions}
    \u(\x,t)=\boldsymbol{0}, \quad \partial_{\n}\dd(\x,t)= \boldsymbol{0}
    \quad\mbox{ for $(\x,t)\in\Sigma$,}
\end{equation}
and the initial conditions
\begin{equation}\label{initial-conditions}
    \dd(\textit{\textbf{x}},0)=\dd_0(\textit{\textbf{x}}),\quad \quad
    \u(\textit{\textbf{x}},0)=\u_0(\textit{\textbf{x}})\quad\mbox{ for
    $\x\in\Om$.}
\end{equation}

In system \eqref{LC}, $\u:\overline Q\to\mathbb{R}^M$ is the velocity of 
the liquid crystal flow, $p:\overline Q\to\mathbb{R}$ is the pressure, and 
$\dd:\overline Q\to\mathbb{R}^M$ is the  orientation of the molecules. 
The physical parameters $\nu$, $\lambda$ and $\gamma$ stand for positive 
constants which represent viscosity, elasticity, and relaxation time, 
respectively. The geometrical parameter $\beta\in [-1,0]$ is a constant 
associated with the aspect ratio of the ellipsoid particles. For instance, 
$\beta =-1/2, -1$ and $0$, corresponds to spherical, rod-like and disk-like 
liquid crystal molecules \cite{Jeffery1922},  respectively. Moreover, 
$\varepsilon>0$ is the penalty parameter and $\f(\varepsilon,\dd)$ is the 
penalty function related to the constraint $|\dd|=1$. This penalty term can 
also be physically meaningful and represents a possible extensibility 
of molecules. It is  defined by
\begin{equation}\label{Pen_func}
\f(\varepsilon,\dd)= \frac{1}{\varepsilon^{2}}\left(|\dd|^2-1\right)\dd.
\end{equation}
It should be noted that 
$\f(\varepsilon,\dd)= \nabla_\dd F(\varepsilon,\dd)$, for all $\dd\in \R^M$, 
for the following scalar potential function
$$
F(\varepsilon,\dd)=\frac{1}{4\varepsilon^2}(|\dd|^2-1)^2.
$$

The following energy law for system \eqref{LC} holds under some regularity
 assumptions for $\u$ and $\dd$ \cite{Lin2007, Liu2007}:
\begin{equation}\label{Ginzburg-Landau-Energy}
\frac{d}{dt}{\mathcal E}(\u, \dd)+\nu \int_\Omega |\nabla\u|^2+
\lambda\gamma \int_\Omega |-
\delta\dd+\f(\varepsilon,\dd)|^2= 0,
\end{equation}
where
\begin{equation}\label{energy}
{\mathcal E}(\u, \dd)=\frac{1}{2} \int_\Omega|\u|^2+\frac{\lambda}{2} 
\int_\Omega |\nabla\dd|^2+ \lambda\int_{\Omega}  F(\varepsilon,\dd).
\end{equation}
The energy \eqref{energy} expresses the competition between the kinetic 
and elastic energies. 
It should be noted that \eqref{Ginzburg-Landau-Energy} is independent of 
$\varepsilon$ and $\beta$. Moreover, it provides a uniform bound for 
$F(\varepsilon, \dd)$ which leads to the unit sphere constraint in the 
limit as $\varepsilon\to0$.

The mathematical structure of system \eqref{LC} consists of the Navier-Stokes
 equations with some extra stress tensors taking into account the geometric 
 of the molecules coupled with a gradient flow equation similar to that of 
 harmonic maps into the sphere. The numerical resolution of system \eqref{LC} 
 is a nontrivial task due to the nonlinear nature of the system, the coupling 
 terms between orientation and flow and the presence of the incompressibility
  constraint.   

It is well known that time discretizations for the Navier--Stokes equations 
based on implicit time strategies give rise to highly time-consuming linear 
solver steps due to the coupled computation of both velocity and pressure. Furthermore, choices of stable finite element pairings are restricted by the 
inf--sup constraint. As a result, projection-based time-splitting techniques
 \cite{Guermond2006} turn to be a very favorable alternative to cutting 
 down the computational cost of current iteration by successively updating 
 velocity and pressure. It is obvious that such a strategy is desirable 
 to solve system \eqref{LC} in order to decouple the pressure computation 
 from  the velocity field as well as the computation of the director field. 
 Instead, inf-sup conditions can be avoided by adding a pressure stabilizing 
 term at the projection step, such that equal interpolation spaces for 
 velocity and pressure are allowed. The stabilizing term is devised in 
 such a way that the stability and convergence rate are not compromised
  \cite{Burman2008}.      

There is an extensive literature on the mathematical analysis of finite 
element methods approximating system \eqref{LC} without stretching effect, 
which corresponds to neglecting all the terms involving the paramter $\beta$. 
The interested reader is referred to the works 
\cite{Liu2000, Liu2002, Becker2008, Girault2011, Badia2011, Badia2011_ii, Guillen2013, Cabrales2015}. Very little has appeared for system \eqref{LC} 
in the context of finite elements. In \cite{Lin2007}, Lin, Liu and Zhang 
presented a modified Crank-Nicolson scheme for which a discrete energy law 
is derived. As a solver, the author devised a fixed iterative method which 
gave rise to a matrix being symmetric and independent of time and the 
number of the fixed point iterations at each iteration. In \cite{Liu2007}, 
Liu, Lin and Zhang proposed a numerical algorithm based on a semi-implicit 
BDF2 rotational pressure-correction method for an axi-symmetric domain for 
studying the annihilation of a hedgehog-antihedgehog pair of defects. 
An extra term related to $\f(\varepsilon, \dd)$ was added so that the proposed 
scheme was somewhat unconditional, but no energy estimate was provided.  

The goal of this paper is to use the above-mentioned techniques for the 
Navier-Stokes equations so as to be able to potentially enhance the 
performance of previous algorithms for system \eqref{LC}. In particular, 
we look for a numerical scheme  using low-order finite-elements, which is 
linear at each time step and unconditionally stable and decouples the computation 
of the all primary variables.  Moreover, we investigate the interplay of the 
flow of a nematic fluid and the geometric shape of its molecules in several 
numerical experiments.

The outline of the rest of this paper is the following. In section 2 we establish the function spaces, the notation and the hypotheses used in the work. Then the new numerical method are introduced. In  section 3,  we prove a priori estimates for the algorithm which provides the unconditional energy-stability property. The paper finishes with section 4, where some details about the implementation and  numerical simulations that illustrate the performance of the scheme concerning the evolution of singularities are presented.

\section{The numerical algorithm}

\subsection{Notation}

For $1\le p \le \infty$, $L^p(\Omega)$ denotes the space of $p$th-power 
integrable real-valued functions defined on $\Omega$  for the Lebesgue 
measure. This space is a Banach space endowed with the norm
$\|v\|_{L^p(\Omega)}=(\int_{\Omega}|v(\x)|^p\,{\rm d}\x)^{1/p}$ for 
$1\le p <\infty$ or
$\|v\|_{L^\infty(\Omega)}={\rm ess}\sup_{\x\in \Omega}|v(\x)|$ for $p=\infty$. 
In particular, $L^2(\Omega)$ is a Hilbert space with 
the inner product
$$
\left(u,v\right)=\int_{\Omega}u(\x)v(\x){\rm d}\x,
$$
and its norm is simply denoted by $\|\cdot\|$.
For $m$ a non-negative integer, we denote the classical Sobolev spaces as 
$$
H^{m}(\Omega) = \{v \in
L^2(\Omega)\,;\, \partial^k v \in L^2(\Omega)\ \forall ~ |k|\le m\},
$$ 
associated to the norm 
$$
\|v\|_{H^{m}(\Omega)} =\left(\sum_{0\le |k| \le m} \|\partial^k
v\|^2\right)^{1/2}\,,
$$ 
where $k = (k_1,\ldots,k_M)\in{\mathds{N}^M}$ is a
multi-index and $|k| = \sum_{i=1}^M k_i$, which is 
a Hilbert space with the obvious inner
product. We will use boldfaced letters 
for spaces of vector functions and their elements, e.g. $\LL^2(\Omega)$ in 
place  of $L^2(\Omega)^M$.

Let $\mathcal{D}(\Omega)$ be the space of infinitely times
differentiable functions with compact support on $\Omega$. 
The closure of ${\mathcal D}(\Omega)$ in
$H^{m}(\Omega)$ is denoted by $H^{m}_0(\Omega)$. 
We will also make use of the following space of vector fields:
$$
\VV=\{\vv\in \boldsymbol{\mathcal{D}}(\Omega) : \nabla\cdot\vv=0 \mbox{ in } \Omega \}. 
$$
We denote by $\HH$ and $\VV$, the closures of $\boldsymbol{\mathcal{V}}$, 
in the $\LL^2(\Omega)$- and $\HH^1(\Omega)$-norm, respectively,   
which are characterized (for $\Omega$ being Lipschitz-continuous) by 
(see \cite{Temam2001})
\begin{eqnarray*}
\HH&=& \{ \u \in \LL^2(\Omega) : \nabla\cdot\u =0 \mbox{ in } 
\Omega, \u\cdot\boldsymbol{n} = 0 \hbox{ on }
\partial\Omega \},\\
\VV&=& \{\u \in \HH^1(\Omega) : \nabla\cdot\u =0 \mbox{ in } 
\Omega, \u = \boldsymbol{0} \hbox{ on } \partial\Omega \},
\end{eqnarray*}
where $\n$ is the outward normal to $\Omega$ on $\partial \Omega$.    
Finally, we consider 
$$
L^2_0(\Omega)= \{ p \in L^2(\Omega) : \ \int_\Omega
p(\x)\, d\x =0 \}.
$$

\subsection{Hypotheses}
Herein we introduce the hypotheses that will be required along this work.
\begin{enumerate}

%\item [(H0)] We suppose that $\u_0\in \HH$, and
%$\dd_0\in \HH^1(\Omega)$ such that  $|\dd_0|=1$ in $\Omega$.

\item [(H1)] Let $\Omega$ be a bounded domain of $\R^M$ 
 with a polygonal or polyhedral Lipschitz-continuous boundary.

\item[(H2)] Let $\{{\mathcal T}_{h}\}_{h>0}$  be a family of regular, 
quasi-uniform subdivisions of 
$\overline{\Om}$ made up of triangles or quadrilaterals  in two dimensions 
and tetrahedra or hexahedra in three dimensions, so that 
$\overline \Omega=\cup_{K\in {\mathcal T}_h}K$. 

\item [(H3)] Assume three sequences of finite-dimensional spaces 
$\{\dd_h\}_{h>0}$, $\{\VV_h\}_{h>0}$ and $\{Q_h\}_{h>0}$ associated with 
$\{{\mathcal T}_{h}\}_{h>0}$ such that $\DD_h\subset \HH^1(\Omega)$, 
$\VV_h\subset \HH^1_0(\Omega)$ and $P_h\subset H^1(\Omega)\cap L^2_0(\Omega)$. 
Also,  consider an extra finite-element space $\{\WW_h\}_{h>0}$ with 
$\WW_h\subset \LL^2(\Omega)$.   
\item [(H4)]  Suppose that $(\u_0, \dd_0)\in \HH\times \HH^1(\Omega)$ with 
$|\dd_0|\le1$ a.e. in $\Omega$.
\end{enumerate}
\medskip

In particular, hypothesis $\rm (H3)$ allows us to consider equal-order 
finite-element spaces for velocity and pressure. For instance,  let  
$\mathcal{P}_1(K)$ be the set of linear polynomials on a triangle or 
tetrahedron $K$. Thus the space of continuous, piecewise polynomial 
functions  associated to ${\mathcal T}_h$  is denoted as
$$
X_h = \left\{ v_h \in {C}^0(\overline\Omega) \;:\; 
v_h|_K \in \mathcal{P}_1(K), \  \forall K \in \mathcal{T}_h \right\},
$$
and the set of piecewise constant functions as 
$$
Y_h =\{ w_h \in L^\infty(\Omega) \;:\; w_h|_K \in \mathbb{ R},
\ \forall K\in {\mathcal T}_h\}.
$$
We choose the following continuous finite-element spaces 
$$\dd_h=\XX_h,\quad \VV_h=\XX_h\cap\HH^1_0(\Omega)\quad\hbox{and}\quad 
P_h=X_h\cap L^2_0(\Omega),
$$
  for approximating the director, the velocity and the pressure, respectively.
Additionally,  we select the extra discontinuous finite-element $\WW_h=\YY_h$
 to be the   space for an auxiliary variable related to the vector director.

Observe that our choice of the finite-element spaces for velocity and  
pressure does not satisfy the discrete inf-sup condition
\begin{equation}\label{LBB-condition}
\|p_h\|_{L^2_0(\Om)}\le \alpha \sup_{\vv_h\in\VV_h\setminus
\{0\}}\frac{\Big(q_h,\nabla\cdot \vv_h\Big)}{\|\vv_h\|_{H^1(\Om)}}
\quad \forall\, p_h\in P_h,
\end{equation}
for $\alpha>0$  independent of $h$. 

The following proposition is concerned with an interpolation operator 
$I_h$ associated with the space $\dd_h$. In fact, we can think of $I_h $ as 
the Scott-Zhang interpolation operator, see \cite{Scott-Zhang}.

\begin{proposition} Assuming hypotheses $\rm (H1)$-$\rm (H3)$, there 
exists $I_h:\HH^1(\Omega)\to \dd_h$ an interpolation operator 
 satisfying 
\begin{equation}\label{interp_error_Ih}
\|\dd-I_h \dd\| \leq C_{app}\, h\|\nabla\dd\|\quad
\forall\,\dd\in\HH^1(\Omega),
\end{equation}
%and 
\begin{align}
\label{stabLinf}
\|I_h\dd\|_{\LL^\infty(\Omega)}\le C_{sta} \|\dd\|_{\LL^\infty(\Omega)} 
\quad\forall\,\dd\in \LL^\infty(\Omega),
\\
\label{stabH1}
\|I_h\dd\|_{\HH^1(\Omega)}\le C_{sta} \|\dd\|_{\HH^1(\Omega)} 
\quad\forall\,\dd\in \HH^1(\Omega),
\end{align}
where $C_{app}>0$ and $C_{sta}>0$ are constants independent of $h$. 
\end{proposition}

\subsection{Description of the scheme}\label{seccionAlgoritmo} As explained 
in the introduction, we aim to construct  a numerical solution to system 
\eqref{LC} that, at each time step, one only needs to solve a sequence of 
decoupled elliptic equations for director, velocity and pressure. 
In particular, the linear systems associated to director and pressure are 
symmetric; therefore, scalable parallel solvers can be defined. Instead, 
the linear system associated to velocity is block diagonal, which means 
that each component of velocity can be computed in parallel.     
This makes our time-splitting be very appealing for high performance computing.  
  
The starting point to design our time-splitting method is the non-incremental
 velocity-correction method for the Navier-Stokes equations. Moreover, 
 it is also rather standard to take an essentially quadratic  truncated potential 
 $\widetilde F(\varepsilon,\dd)$ instead of the ``quartic'' potential 
 $F(\varepsilon,\dd)$.  To be more precise, one considers
\begin{equation}\label{Truncated-Pontential-fun}
\widetilde F(\varepsilon,\dd)=\frac{1}{\varepsilon^2}
\begin{cases}
\displaystyle
\frac{1}{4}(|\dd|^2-1)^2, &\mbox{ if } |\dd|\le 1,\\
%& \\
\displaystyle
(|\dd|-1)^2, &\mbox{ if } |\dd|> 1,
\end{cases}
\end{equation}
for which
$$
\widetilde\f(\varepsilon,\dd)=\nabla_\dd \widetilde F(\varepsilon,\dd)= 
 \frac{1}{\varepsilon^2}
\begin{cases}
\displaystyle
(|\dd|^2-1) \dd, &\mbox{ if } |\dd|\le 1,\\
%& \\
\displaystyle
2 (|\dd|-1)\frac{\dd}{|\dd|}, &\mbox{ if } |\dd|> 1.
\end{cases}
$$

Let $N\in\mathds{N}$ and let $k=T/N$ denote the time-step size. To start 
up the sequence of approximation solutions, we consider  
\begin{equation}\label{initial-d0}
\dd_{0h}=\mathcal{I}_h\dd_0
\end{equation} 
and  $(\u_{0h},p_{0h})\in \VV_h\times P_h$ such that
\begin{subequations}\label{initial-u0}
\begin{empheq}[left=\empheqlbrace]{align}
(\u_{0h}, \bar \u_h)+(\nabla p_{0h}, \bar\u_h )&=(\u_0, \bar\u_h ),
\label{initial-u0eq1}\\
(\nabla\cdot\u_{0h}, \bar p_h)+j( p_{0h}, \bar p_h)&=0,\label{initial-u0eq2}
\end{empheq}
\end{subequations}
for all $(\bar\u_h, \bar p_h)\in \VV_h\times P_h$. It should be noted that 
$\|\u_{0h}\|\le C \|\u_0\|$ holds.
 
Then the algorithm reads as follows. Let 
$(\dd^{n}_h, \u^{n}_h)\in \dd_h\times \VV_{h}$ be 
given. For the $n+1$ time step, do the following steps: 
\begin{enumerate}
\item Find $(\dd^{n+1}_h, \w^{n+1}_h)\in \dd_h\times\WW_h$ satisfying 
\begin{equation}\label{scheme3eq1}
\left\{
\begin{array}{rcl}
\displaystyle
\left(\frac{\dd^{n+1}_h-\dd^n_h}{k},\bar\w_h\right) + (\u^\star_h, (\nabla\dd^n_h)^T \bar\w_h)
- \beta(\u^{\star\star}_h, \nabla\cdot (\bar\w_h (\dd^n_h)^T))&&
\\
-(1+\beta)(\u^{\star\star\star}_h, \nabla\cdot (\dd^n_h \bar\w^T_h) )
+\gamma( \w^{n+1}_h, \bar\w_h)&=&0,%\label{scheme3eq1a}
\\
\displaystyle(\nabla \dd^{n+1}_h, \nabla \bar  \dd_h)+(\widetilde \f_\varepsilon(\dd^n_h)+
\frac{H_F}{2\varepsilon^2}(\dd^{n+1}_h-\dd^n_h), \bar\dd_h)-(\w^{n+1}_h, \bar\dd_h)=0,%\label{scheme3eq1b}
\end{array}
\right.
\end{equation}
for all $(\bar\dd_h, \bar\w_h)\in \DD_h\times\WW_h$, where (see 
\eqref{constanteHf} below)
$$
H_{F}:=(M3^2+(M^2-M)2^2)^{1/2}
$$
and
\begin{eqnarray*}
\u^\star_h&=&\u^n_h+3\lambda\, k\, (\nabla\dd_h^n)^T \w^{n+1}_h,\\
\u^{\star\star}_h&=&\u^{n}_h-3\lambda \beta k \,\nabla\cdot (\w^{n+1}_h (\dd^n_h)^T),
\\
\u^{\star\star\star}_h&=&\u^{n}_h-3\lambda (1+\beta) k \,\nabla\cdot 
(\dd^{n}_h (\w^{n+1}_h)^T).
\end{eqnarray*}
\item Find $p^{n+1}_h\in  P_{h}$ satisfying
\begin{equation}\label{scheme3eq3a}
k(\nabla p^{n+1}_h, \nabla \bar  p_h)+ j (p^{n+1}_h, \bar p_h)
-(\widetilde\u^{n+1}_h,\nabla\bar p_h)=0,
\end{equation}
for all $\bar p_h\in P_h$, with
\begin{eqnarray*}
\widetilde\u^{n+1}_h&=&\frac{\u^\star_h+\u^{\star\star}_h+\u^{\star\star\star}_h}{3}
\\
&=&\u^n_h+\lambda\, k \, \Big((\nabla\dd_h^n)^T \w^{n+1}_h - 
 \beta \nabla\cdot (\w^{n+1}_h (\dd^n_h)^T)
 -(1+\beta)  \nabla\cdot (\dd^{n}_h (\w^{n+1}_h)^T) \Big)
\end{eqnarray*}
and
$$
j(p^{n+1}_h, \bar p_h)=S \frac{1}{\nu} (p^{n+1}_h-\pi_0( p^{n+1}_h), 
\bar p_h-\pi_0( \bar p_h)),
$$
where $S>0$ is an algorithmic constant, and  $\pi_0$  is the 
$L^2$-orthogonal projection operator onto $Y_h$.

\item Find $\u_n^{n+1}\in \VV_h$ satisfying 
\begin{equation}\label{scheme3eq2}
\begin{array}{l}
\displaystyle \left(\frac{\u^{n+1}_h-\u^n_h}{k},\bar\u_h\right)+
c(\u^n_h, \u^{n+1}_h,\bar\u_h)
+\nu(\nabla\u^{n+1}_h,\nabla\bar\u_h)
+(\nabla p_h^{n+1}, \bar\u_h )\\
+(- \lambda (\nabla\dd^n_h)^T \w^{n+1}_h+\lambda\beta \,\nabla\cdot 
(\w^{n+1}_h (\dd^n_h)^T)+\lambda (1+\beta) \,\nabla\cdot 
(\dd^{n}_h (\w^{n+1}_h)^T), \bar\u_h)= 0, 
\end{array}
\end{equation}
for all $\bar\u_h\in \VV_h.$
\end{enumerate}

To enforce the skew-symmetry of the trilinear convective term in 
\eqref{scheme3eq2}, we have defined
$$
c(\u_h,\vv_h,\w_h)=((\u_h\cdot\nabla)\vv_h, \w_h)+
\frac{1}{2} (\nabla\cdot \u_h, \vv_h\cdot	\w_h)
$$   
for all $\u_h, \vv_h, \w_h\in \VV_h$. Thus, 
$c(\u_h,\vv_h,\vv_h)=0$ for all $\u_h, \vv_h\in \VV_h$.  

Since scheme \eqref{scheme3eq1}-\eqref{scheme3eq2} is linear, it suffices 
to prove its uniqueness, which follows easily by comparing two solutions
 \cite{Cabrales2015}.

The main characteristic of scheme \eqref{scheme3eq1}-\eqref{scheme3eq2} 
is that the approximations $(\dd^{n+1}_h, \w^{n+1}_h)$, $p^{n+1}_h$ and  
$\vv^{n+1}_h$ are performed successively. The use of the auxiliary variable 
$\w^{n+1}_h$ is just to be able to derive a priori energy estimates, 
although the computation of $\w^{n+1}_h$ can be avoided  as will be seen 
in section \ref{Numerical_results}.       

Concerning the stabilizing term $j(\cdot,\cdot)$, other choices are 
feasible if  one wants to improve the spatial convergence rate: 
$$
j(p_h,\bar p_h)= S \frac{h^2}{\nu} (\nabla p_h-\pi(\nabla p_h), \nabla \bar p_h-
\pi(\nabla \bar p_h)),
$$
where $\pi$ could be the $\LL^2(\Omega)$-orthogonal protection operator 
onto $\YY_h$ \cite{Codina2001} or the Scott-Zhang operator into $\YY_h$
 \cite{Badia2012}. The latter is more appealing since no auxiliary 
 variable is required to compute it.   
 
\section{A priori energy estimates}

Let us begin by noting that the  $ij$-component of the Hessian matrix of 
the truncated potential
$\widetilde F(\varepsilon,\dd)$ with  respect to $\dd$ is given by
%\subsection{Truncate potential} 
%and 
$$
H_\dd \widetilde F(\varepsilon,\dd)_{ij}= \frac{1}{\varepsilon^2}
\begin{cases}
\displaystyle 2 d_i d_j + (|\dd|^2-1)\delta_{ij} , &\mbox{ if } |\dd|\le 1,
\\
%& \\
\displaystyle
2\frac{d_id_j}{|\dd|^3} + 
 2\frac{|\dd|-1}{|\dd|} \delta_{ij}, &\mbox{ if } |\dd|> 1.
\end{cases}
$$
%where   $H_\dd\widetilde F_\varepsilon$ stands for the Hessian matrix of
%$\widetilde F_\varepsilon$ with 
%respect to $\dd$.
We thus have 
% of  $\widetilde F(\varepsilon,\dd)$ with center $\dd^n$ 
%evaluated at $\dd^{n+1}$ gives 
\begin{equation}\label{approx-potential}
\begin{array}{rl}
\widetilde F(\varepsilon,\dd^{n+1})-\widetilde F(\varepsilon,\dd^n)&=
\nabla_\dd \widetilde F(\varepsilon,\dd^n) \cdot (\dd^{n+1}-\dd^n)\\
&+\displaystyle
\frac{1}{2}(\dd^{n+1}-\dd^n)^TH_\dd \widetilde 
F(\varepsilon,\dd^{n+\theta})(\dd^{n+1}-\dd^n),
\end{array}
\end{equation}
where $\dd^{n+\theta}=\theta \dd^{n+1}+(1-\theta)\dd^n$ for some 
$\theta\in (0,1)$. Since $\widetilde F(\varepsilon,\cdot)$ is 
essentially quadratic, each component of 
the associated Hessian is uniformly bounded as 
$$
\|H_\dd  \widetilde F(\varepsilon,\cdot)_{i j}\|^2_{L^\infty( \R^M)}\le 
\frac{1}{\varepsilon^2} (2+\delta_{ij}).
$$
Hence, the Frobenius norm is bounded as  
\begin{equation}\label{constanteHf}
\left(\sum_{i,j}^{M}\|H_\dd  \widetilde 
F(\varepsilon,\cdot)_{i j}\|^2_{L^\infty( \R^M)}\right)^{1/2}
\le \frac{1}{\varepsilon^2}H_F,\mbox{ where }H_{F}=(M3^2+(M^2-M)2^2)^{1/2}.
\end{equation}
 In particular, using the consistence of the Frobenius norm gives
\begin{equation}\label{bound-frobenius}
\frac{1}{2}(\dd^{n+1}-\dd^n)^TH_\dd \widetilde 
F(\varepsilon,\dd^{n+\theta})(\dd^{n+1}-\dd^n)\le 
\frac{H_F}{2\varepsilon^2}|\dd^{n+1}-\dd^n|^2. 
\end{equation}

Consequently,  by adding to  $\widetilde\f(\varepsilon,\dd^n)$  a large 
enough first-order linear dissipation term, 
$\displaystyle\widetilde\f(\varepsilon,\dd^n)+
\frac{H_F}{2\varepsilon^2}(\dd^{n+1}-\dd^n)$, we have, from 
\eqref{approx-potential} and 
\eqref{bound-frobenius},  
\begin{equation}\label{uncond-stab}
\left(\widetilde\f(\varepsilon,\dd^n)
+\frac{H_F}{ 2 \varepsilon^2}(\dd^{n+1}-\dd^n)\right)\cdot 
(\dd^{n+1}-\dd^n) \ge \widetilde 
F(\varepsilon,\dd^{n+1})-\widetilde F(\varepsilon,\dd^n).
\end{equation}
This inequality will play an essential role for the energy-stability 
of scheme \eqref{scheme3eq1}.
To prove this, we denote the discrete energy as
$$
{\mathcal E}(\u_h, \dd_h)=\frac{1}{2}\|\u_h\|^2+\frac{\lambda}{2} \|\nabla\dd_h\|^2+ 
\lambda\int_{\Omega} \widetilde F(\varepsilon,\dd_h).
$$
%
%By using the following identities
%\begin{eqnarray*}
%&\lambda\nabla\cdot((\nabla\dd)^T\nabla\dd)=\lambda\nabla
%\left(\frac{1}{2}|\nabla\dd|^2+
%F(\varepsilon,\dd)\right)-\lambda(\nabla\dd)^T(\f(\varepsilon,\dd)-\delta \dd),&\\
%&\left((\u\cdot\nabla)\dd\right) \cdot (\f(\varepsilon,\dd)-\delta\dd)=
%\left(\nabla\dd)^T(\f(\varepsilon,\dd)-\delta\dd)\right)\cdot\u,&
%\end{eqnarray*}
%one can prove the following energy law for system 
%\eqref{LC}-\eqref{initial-conditions}:
%\begin{equation}\label{LC-energy}
%\frac{d}{dt}\left(\frac{1}{2}\|\u\|^2+\frac{\lambda}{2}\|\nabla\dd
%\|^2 + \lambda \int_{\Omega}F(\varepsilon,\dd)\right)
%+\nu\|\nabla\u\|^2+\lambda\gamma \|\f(\varepsilon,\dd)-\delta\dd\|^2= 0.
%\end{equation} 
%
We are now in a position to prove the following result concerning 
a local-in-time discrete energy estimate. 
\begin{lemma}\label{le:induction} Under hypotheses $(\rm H1)$--$(\rm H4)$, 
it follows that, for any $k>0$, $h>0$ and $\varepsilon>0$, the 
corresponding solution
$(\u^{n+1}_h,p^{n+1}_h,\dd^{n+1}_h,\w^{n+1}_h)$ of scheme 
\eqref{scheme3eq1}-\eqref{scheme3eq2} satisfies the following inequality:
\begin{equation}\label{induction}
\begin{array}{rl}
\displaystyle
{\mathcal E}(\u^{n+1}_h,\dd^{n+1}_h)-{\mathcal E}(\u^{n}_h,\dd^{n}_h) 
+k\left(\nu\|\nabla\u_h^{n+1}\|^2 
+\lambda\gamma \|\w^{n+1}_h\|^2\right)
% \displaystyle
+\frac{\lambda}{2}\|\nabla(\dd^{n+1}_h-\dd^n_h)\|^2&\\
\displaystyle
+\frac{1}{2}\Big(\|\u_h^{n+1}-\widehat\u^{n+1}_h\|^2
+\|\widehat\u_h^{n+1}-\widetilde\u^{n+1}_h\|^2 \Big)
+k\,j(p^{n+1}_h, p^{n+1}_h)&
\\
 \displaystyle
+\frac{1}{2}\frac{\|\widetilde\u^{n+1}_h-\u^\star_h\|^2
+ \|\widetilde\u^{n+1}_h-\u^{\star\star}_h\|^2
+\|\widetilde\u^{n+1}_h-\u^{\star\star\star}_h\|^2}{3}&
\\
 \displaystyle
+\frac{1}{2}\frac{
\|\u^\star_h-\u^n_h\|^2
+\|\u^{\star\star}_h-\u^n_h\|^2
+\|\u^{\star\star\star}_h-\u^n_h\|^2}{3}
 &\le 0 .
\end{array}
\end{equation}

\end{lemma}
\begin{proof} 
We take $\bar\w_h=\lambda\,k\,\w^{n+1}_h$ in $\eqref{scheme3eq1}_a$ 
and $\bar\dd_h=\lambda(\dd^{n+1}_h-\dd^n_h)$ 
in $\eqref{scheme3eq1}_b$ and use  \eqref{uncond-stab} to obtain  
\begin{equation}\label{aux4}
\begin{array}{rl}
\displaystyle
\frac{\lambda}{2}\left(\|\nabla\dd^{n+1}_h\|^2-\|\nabla\dd^n_h\|^2
+\widetilde F(\varepsilon,\dd^{n+1}_h)-\widetilde F(\varepsilon,\dd^n_h)\right)
+\lambda\,\gamma \,k\|\w^{n+1}_h\|^2&\\
\displaystyle
+\frac{\lambda}{2}\|\nabla(\dd^{n+1}_h-\dd^{n}_h)\|^2
+\lambda\, k ((\u^\star_h,(\nabla\dd^{n}_h)^T\w^{n+1}_h)
\\
- \lambda\, k\,\beta\, (\u^{\star\star}_h, \nabla\cdot (\w_h^{n+1} (\dd^n_h)^T))
-\lambda\, k\,(1+\beta)\,(\u^{\star\star\star}_h, \nabla\cdot 
(\dd^n_h (\w^{n+1}_h)^T )&\le 0.
\end{array}
\end{equation}
Next we take $\bar\u_h=k\, \u^{n+1}_h$, as a test function, in \eqref{scheme3eq2}
and introduce the auxiliary velocity  
$\widehat\u^{n+1}_h= \widetilde\u^{n+1}_h -k\,\nabla p^{n+1}_h$ to obtain 
\begin{equation}\label{ecuacion1}
 \frac{1}{2}\|\u^{n+1}_h\|^2
-\frac{1}{2}\|\utecho^{n+1}_h\|^2
+\frac{1}{2}\|\u_h^{n+1}-\utecho^{n+1}_h\|^2
+\nu\,k \|\nabla \u_h^{n+1}\|^2= 0.
\end{equation}
Now we take  $\bar p_h=p^{n+1}_h$  in \eqref{scheme3eq3a} and use 
the auxiliary velocity $\widehat\u^{n+1}_h$ introduced previously to obtain
\begin{equation}\label{ecuacion2}
\frac{1}{2}\| \widehat\u^{n+1}_h\|^2 - \frac{1}{2}\|\widetilde\u^{n+1}_h\|^2+
\frac{1}{2}\| \widehat\u^{n+1}_h -  
\widetilde\u^{n+1}_h \|^2+k\,j(p^{n+1}_h, p^{n+1}_h)=0.
\end{equation}

From the definition of $\widetilde\u^{n+1}_h$ in \eqref{scheme3eq3a}, we write
$$
\frac{\widetilde\u^{n+1}_h-\u^\star_h}{3}
+\frac{\widetilde\u^{n+1}_h-\u^{\star\star}_h}{3}
+\frac{\widetilde\u^{n+1}_h-\u^{\star\star\star}_h}{3}=0.
$$
Hence, 
\begin{equation}\label{ecuacion3}
 \frac{1}{2}\|\widetilde\u^{n+1}_h\|^2
 -\frac{1}{2}\frac{\|\u^\star_h\|^2+\|\u^{\star\star}_h\|^2+
 \|\u^{\star\star\star}_h\|^2}{3}
 +\frac{1}{2}\frac{\|\widetilde\u^{n+1}_h-\u^\star_h\|^2
 +\|\widetilde\u^{n+1}_h-\u^{\star\star}_h\|^2
 +\|\widetilde\u^{n+1}_h-\u^{\star\star\star}_h\|^2}{3}=0.
\end{equation}
Moreover, from the definitions of $\u^\star_h$, $\u^{\star\star}_h$ 
and $\u^{\star\star\star}_h$ in $\eqref{scheme3eq1}_a$, we deduce the 
following equalities:
\begin{eqnarray}
 \frac{1}{6}\| \u^\star_h \|^2
-\frac{1}{6}\| \u^n_h \|^2
+\frac{1}{6}\| \u^\star_h -\u^n_h\|^2
- \lambda\, k \, ((\nabla\dd^n_h)^T \w^{n+1}_h,\u^\star_h)&=& 0,\label{ecuacion4}\\
%\end{equation*}
%\begin{equation*}
 \frac{1}{6}\| \u^{\star\star}_h \|^2
-\frac{1}{6}\| \u^n_h \|^2
+\frac{1}{6}\| \u^{\star\star}_h -\u^n_h\|^2
+ \lambda\, k \,\beta\, (\u^{\star\star}_h, \nabla\cdot (\w_h^{n+1} (\dd^n_h)^T))
&=& 0,\label{ecuacion5}\\
%\end{equation*}
%and
%\begin{equation*}
 \frac{1}{6}\| \u^{\star\star\star}_h \|^2
-\frac{1}{6}\| \u^n_h \|^2
+\frac{1}{6}\| \u^{\star\star\star}_h -\u^n_h\|^2
+ \lambda\, k \,(1+\beta)\,(\u^{\star\star\star}_h, \nabla\cdot 
(\dd^n_h (\w^{n+1}_h)^T )
&=& 0,\label{ecuacion6}
\end{eqnarray}
Adding equalities \eqref{ecuacion1}-\eqref{ecuacion6} leads to 
\begin{equation}\label{aux1}
\begin{array}{rl}
\displaystyle
 \frac{1}{2}\|\u^{n+1}_h\|^2 - \frac{1}{2}\|\u^n_h\|^2
 +\nu\,k \| \nabla\u_h^{n+1}\|^2
 + \lambda\, k \, ((\nabla\dd^n_h)^T \w^{n+1}_h,\u^\star_h)&\\
 \displaystyle
+\frac{1}{2}\left(\|\u_h^{n+1}-\utecho^{n+1}_h\|^2
+\| \widehat\u^{n+1}_h - \widetilde \u^{n+1}_h \|^2\right)
+k\,j(p^{n+1}_h, p^{n+1}_h)& \\
 \displaystyle
+\frac{1}{2}\frac{\|\widetilde\u^{n+1}_h-\u^\star_h\|^2
+\|\widetilde\u^{n+1}_h-\u^{\star\star}_h\|^2
+\|\widetilde\u^{n+1}_h-\u^{\star\star\star}_h\|^2}{3}& \\
 \displaystyle
+\frac{1}{2}\frac{\|\u^\star_h-\u^{n}_h\|^2+\|\u^{\star\star}_h-\u^{n}_h\|^2
+\|\u^{\star\star\star}_h-\u^{n}_h\|^2}{3}&\\
 \displaystyle
- \lambda\, k \, ((\nabla\dd^n_h)^T \w^{n+1}_h,\u^\star_h)
+ \lambda\, k \,\beta\, (\u^{\star\star}_h, \nabla\cdot (\w_h^{n+1} (\dd^n_h)^T))&\\
 \displaystyle
+ \lambda\, k \,(1+\beta)\,(\u^{\star\star\star}_h, \nabla\cdot
 (\dd^n_h (\w^{n+1}_h)^T )
&= 0.
\end{array}
\end{equation}
Finally, we add (\ref{aux4}) and  \eqref{aux1}; thus the terms  
$((\nabla\dd^n_h)^T\w^{n+1}_h,\u^\star_h)$, 
$ (\u^{\star\star}_h, \nabla\cdot (\w_h^{n+1} (\dd^n_h)^T))
$ and 
$(\u^{\star\star\star}_h, \nabla\cdot (\dd^n_h (\w^{n+1}_h)^T )$
cancel out, and hence \eqref{induction} holds. This finishes the proof.
\end{proof}

Now, it is not difficult to extend the previous local-in-time discrete energy estimate to a 
global-in-time one.
\begin{theorem}\label{th:stability}
Assume that  $\rm (H0)$-$\rm (H4)$ are satisfied. The discrete solution 
$\{(\u^n_h, \dd_h^n, \w^{n}_h)\}_{n=0}^N$ of scheme 
\eqref{scheme3eq1}-\eqref{scheme3eq2} satisfies  
\begin{equation}\label{global-estimate}
\max_{r\in\{0,\cdots, N-1\}}\left\{{\mathcal E}(\u^{r+1}_h,\dd^{r+1}_h)
 +k\sum_{n=0}^r \left(\nu
\|\nabla\u_h^{n+1}\|^2 +\lambda\gamma \|\w^{n+1}_h\|^2\right)\right\}
\le {\mathcal E}(\u_{0h},\dd_{0h}).
\end{equation}
\end{theorem}
\begin{proof} The proof follows easily from Lemma \ref{le:induction} 
and by summing over $n$.  
\end{proof}

It remains to prove  that  the initial energy ${\mathcal E}(\u_{0h},\dd_{0h})$ 
is bounded independent of $(h,k,\varepsilon)$.
\begin{lemma}\label{le:initial-bound}
 Assume that hypotheses $\rm (H1)$-$\rm (H4)$ hold.  If $(h,\varepsilon)$ 
 are chosen satisfying 
\begin{equation}\label{const-h-eps}
\frac{h}{\varepsilon}\le K,
\end{equation} for some constant $K>0$, then 
\begin{equation}\label{initial-bound}
{\mathcal E}(\u_{0h}, \dd_{0h})\leq C_0
\end{equation}
for the initial approximations $(\u_{0h}, \dd_{0h})$ defined in 
\eqref{initial-d0} and \eqref{initial-u0}.
\end{lemma}
\begin{proof} We take $\bar\u=\u_{0h}$ and $\bar p_h=p_{0h}$ as 
test functions into \eqref{initial-u0} to obtain
\begin{equation}\label{aux8}
\frac{1}{2}\|\u_{0h}\|^2+j(p_{0h}, p_{0h})\le \frac{1}{2}\|\u_0\|^2.
\end{equation}

Moreover, from \eqref{stabH1}, we have
\begin{equation}\label{aux9}
\|\dd_{0h}\|_{\HH^1(\Omega)}\le C \|\dd_0\|_{\HH^1(\Omega)}.
\end{equation}
Now, we bound as in \cite{Guillen2013}, 
\begin{eqnarray}
\displaystyle
 \int_{\Omega} F(\varepsilon,\dd_{0h})&\le&
\displaystyle \frac{1}{\varepsilon^2} \int_{\Omega} (|\dd_{0h}|^2-|\dd_0|^2)^2
=\frac{1}{\varepsilon^2}\int_{\Omega}(|\dd_{0h}+\dd_0||\dd_{0h}-\dd_{0}|)^2
\nonumber\\
&\le & \displaystyle \|\dd_{0h}
+\dd_0\|^2_{\LL^\infty(\Omega)} \|\dd_{0h}-\dd_0\|^2\le 
C \, \frac{h^2}{\varepsilon^2}\|\dd_0\|^2_{\HH^1(\Omega)}, \label{aux10}
\end{eqnarray}
where \eqref{interp_error_Ih} and \eqref{stabLinf}   has been applied. 
Combining (\ref{aux10}) with (\ref{aux8}) and (\ref{aux9}) together with 
(\ref{const-h-eps}), we obtain \eqref{initial-bound}.
\end{proof}

\begin{remark}  Hypothesis (\ref{const-h-eps}) is rarely explicitly 
mentioned in numerical papers based on algorithms using the penalty 
approach but it is required to guarantee a priori energy estimates 
independent of $\varepsilon$. It seems that this condition is overlooked.
 Nevertheless, it is important to underline that the constraint 
 (\ref{const-h-eps}) for $(h,\varepsilon)$ comes only from the 
 approximation of $\dd_0$ in $\DD_h$ but not from the discrete scheme itself.
\end{remark}
\section{Numerical results}\label{Numerical_results}
%K: PARA RCC
%\begin{enumerate}
%\item Reescribir el esquema eliminando $\u^{n}$ como un m\'etodo desacoplado en $(\d^{n+1}_h, \w^{n+1}_h)$, $p^{n+1}_h$ and $\u^{n+1}_h$.
%
%K: ESTO YA ESTA HECHO
%
%\item Eliminar el c\'alculo de $\w^{n+1}$ en funci\'on de $\d^{n+1}$  a nivel algebraico.
%
%RC: ESTO YA ESTA HECHO, FALTA REVISAR.
%
%\item Comparar con las simulaciones num\'ericas de los art\'\i culos  
%\cite{Lin-Liu-Zhang-2007} y \cite{Liu-Shen-Yang-2007} con t\'erminos ``stretching".
%
%Art\'\i culo \cite {Lin-Liu-Zhang-2007}
%\begin{enumerate}
%\item Comparaci\'on (tiempo aniquilaci\'on) vs $\beta$.
%\item Comparaci\'on con campo de velocidad rotando. Sin t\'erminos 
%``stretching" no aniquila y con t\'erminos ``stretching'' si aniquila.
%\end{enumerate}
%Art\'\i culo \cite{Liu-Shen-Yang-2007}.
%\begin{enumerate}
%\item Stretching implica symmetry breaking
%\item $\beta$ decreciendo implica tiempo de aniquilaci\'on creciendo.
%\item $\beta=-0.1$ forma de disco, $\beta=-0.5$ forma esf\'erica, $\beta=-0.9$ forma de barra. 
%\end{enumerate}
%RC: ESTE ITEM ESTA EN PROGRESO.
%\end{enumerate}

From now on, we consider the particular instance of the approximating spaces $\DD_h, \VV_h$  and $P_h$ 
described in $\rm (H3)$. The main objective of this section is to illustrate the stability, efficiency and reliability of scheme \eqref{scheme3eq1}-\eqref{scheme3eq2}. In doing so, we will present some  numerical experiments concerning the annihilation of singularities, and  use the results to check numerically how the stabilization constant $H_F$  must be chosen to assure  the unconditional stability. The numerical  solutions are calculated with a computer program implemented in FreeFem++ \cite{Hecht2012}. 

Before going further, we discuss some implementation issues concerning scheme
\eqref{scheme3eq1}-\eqref{scheme3eq2}.  

\subsection{Implementation issues}\label{sect:Implementation}
%It is clear that the computation of the auxiliary  vector $\w^{n+1}_h$ in \eqref{scheme_eq1} 
%gives rise to an important amount of computer memory and time in comparison with systems
% \eqref{scheme_eq2}) and (\ref{scheme_eq3a}). Thus, Scheme \eqref{scheme_eq1})-\eqref{scheme_eq2}) will become much more efficient if we are able to remove 
% $\w^{n+1}_h$ from \eqref{scheme_eq1}. 
%

Let $N_d=\dim(\DD_h)$,  $N_w=\dim(\WW_h)$ and let $\{\Bphi^d_i\}_{i=1}^{N_d}$ and 
$\{\Bphi^w_i\}_{i=1}^{N_w}$ be the  finite-element bases for $\DD_h$ and 
$\WW_h$,  respectively, 
constructed from  the local basis %$\{\phi_i\}_{i=1}^{I}$ and $\{\psi_l\}_{l=1}^{L}$
 of $X_h$ and $Y_h$, respectively. 
  We describe the following matrices related to $\WW_h$:
\begin{eqnarray*}
&\Mmat_{w,d}=
\left(\int_\Omega \Bphi_i^d\cdot\Bphi_j^w\right)_{ij},\quad
\Mmat_{w}=
\left(\int_\Omega \Bphi_i^w\cdot\Bphi_j^w\right)_{ij},\quad
\displaystyle
\Bmat_{w}=
\Bmat^{\star}_{w}+
\Bmat^{\star\star}_{w}+\Bmat^{\star\star\star}_{w},&
\end{eqnarray*}
where
\begin{eqnarray*}
&\displaystyle\Bmat^{\star}_{w}=
3\lambda k\left( \int_\Omega [ \nabla\dd_h^n]^T\Bphi_i^w\cdot 
[\nabla\dd^n_h]^T \Bphi_j^w\right)_{ij},\quad
\Bmat^{\star\star}_{w}=
3\lambda\beta^2k\left( \int_\Omega\left[\nabla\cdot(\Bphi_i^w[\dd_h^n]^T)\right]
\cdot 
\left[\nabla\cdot ( \Bphi_j^w [\dd^n_h]^T)\right]\right)_{ij},&\\
&\displaystyle\Bmat^{\star\star\star}_{w}=
3\lambda(1+\beta)^2k\left(
\int_\Omega
\left[\nabla\cdot(\dd_h^n[\Bphi_i^w]^T)\right]\cdot 
\left[\nabla\cdot (\dd^n_h [\Bphi_j^w]^T)\right] 
\right)_{ij}.&
\end{eqnarray*}
Also,   we introduce the matrices related to $\DD_h$:
$$
\Mmat_{d}=
\left(\int_\Omega \Bphi_i^d\cdot\Bphi_j^d\right)_{ij},
\quad
\Mmat_{d,w}=
\left(\int_\Omega \Bphi_i^w\cdot\Bphi_j^d\right)_{ij},\quad
\Lmat_d=
\left(\int_\Omega \nabla\Bphi_i^d\cdot\nabla\Bphi_j^d\right)_{ij}.
$$
Moreover, let us denote by $\Wvec \in \mathbb{R}^{N_w}$ and 
$ \Dvec \in \mathbb{R}^{N_d}$
the coordinate vectors, with respect to the fixed bases of the 
finite-element functions $\w_h \in \WW_h$ and  $\dd_h \in \DD_h$, 
respectively. Thus, we can rewrite system 
\eqref{scheme3eq1} as 
\begin{subequations}\label{matrix_version_eq1}
\begin{empheq}[left=\empheqlbrace]{align}
\displaystyle \frac{1}{k}\Mmat_{w,d}\Dvec^{n+1}+ 
\left(\Bmat_{w} +\gamma\Mmat_w\right)\Wvec^{n+1}&=
\frac{1}{k}\Mmat_{w,d}\Dvec^{n} -\Fvec_{w},\label{matrix_version_eq1a}\\
\Lmat_d\Dvec^{n+1}+\frac{H_F}{2\varepsilon^2}\Mmat_{d}\Dvec^{n+1}
-\Mmat_{d,w}\Wvec^{n+1}&=\Fvec,
\label{matrix_version_eq1b}
\end{empheq}
\end{subequations}
where $\displaystyle\Fvec_{w}\in \R^{N_d}$ and $\Fvec\in \R^{N_w}$ 
are defined, respectively, as
\begin{eqnarray*}%\label{matrix_version_eq1c}
&\displaystyle\Fvec_{w}=
\left(\int_\Omega \u_h^n\cdot\left[[\nabla\dd^n_h]^T \Bphi_j^w\right]
-\beta\int_\Omega \u_h^n\cdot  \left[\nabla\cdot (\Bphi_j^w (\dd^n_h)^T)\right]
-(1+\beta)\int_\Omega \u_h^n\cdot \left[\nabla\cdot 
(\dd^n_h (\Bphi_j^w )^T \right] \right)_j,
&\\
&\displaystyle\mbox{and }\Fvec=
\left(\int_\Omega \left[\frac{H_F}{2\varepsilon^2}\dd^n_h-
\tilde\f_\varepsilon(\dd^n_h)\right]\cdot\phi_j^d\right)_j. 
&
\end{eqnarray*} 
By defining $\Emat_w=\Bmat_w+\gamma\Mmat_w$, from  
\eqref{matrix_version_eq1a}, we have
$$
\displaystyle \Wvec^{n+1}=\Emat^{-1}_w \left[
\frac{1}{k}\Mmat_{w,d}\left(\Dvec^{n}-\Dvec^{n+1}\right)
-\Fvec_{w}\right],
$$
where $\Emat_w$ can be seen in two different ways depending on the reordering of
the degrees of freedom of $\Wvec^{n+1}$: (1) A block-diagonal, $M$-by-$M$ 
matrix, which 
is easy to invert by using a block Gauss-Jordan elimination method, or (2) 
an $M$-by-$M$ block, diagonal 
matrix, since the degrees of freedom of two different elements are not coupled, 
which 
is also easily invertible by using block computations. The first approach 
is much more adequate especially 
for legacy code bases, which is our case here. 

Replacing the above equality in equations  \eqref{matrix_version_eq1b}, and after 
some calculations, the resulting system is:
\begin{equation}\label{linear-system}
\left(\Lmat_d+\frac{1}{k}\Mmat_{d,w}\Emat_w^{-1}\Mmat_{w,d}
+\frac{H_F}{2\varepsilon^2}\Mmat_{d}\right)\Dvec^{n+1}
=\Mmat_{d,w}\Emat_w^{-1}
\left[\frac{1}{k}\Mmat_{w,d}\Dvec^{n}-\Fvec_{w}\right]+\Fvec.
\end{equation}
Consequently, we can avoid, at algebraic level, computing the 
auxiliary vector $\Wvec^{n+1}$ by solving directly  
$\eqref{linear-system}$.
Observe that the matrix $\displaystyle\Lmat_d+
\frac{1}{k}\Mmat_{d,w}\Emat_w^{-1}\Mmat_{w,d}
+\frac{H_F}{2\varepsilon^2}\Mmat_{d}$ is the Schur complement of system 
\eqref{matrix_version_eq1} with respect to  $\Emat_w$. Moreover, 
such a matrix is symmetric and positive definite. 

%%%%%%%

\subsection{Annihilation of singularities} In this experiment, we consider 
$\Omega=(-1,1)\times(-1,1)$ and the physical parameters  $\nu=\lambda=\gamma=1$. 
Also, we take  $\beta=-1$, that is, rod-like molecules of the liquid 
crystal are considered.
The discretization, penalization and stabilization parameters are set as
\[
k=0.001, h=0.0790796, \varepsilon=0.05, \mbox{ and } H_F=0.
\]

Given an initial  velocity $\vv_0$, the main objective of the experiment 
is to study the evolution of the singularities of the initial director 
$\dd_0$, that is, 
points  of the computational domain  $\Omega$ where $|\dd_0|=0$. 
In particular, we will 
present two numerical experiences concerning the annihilation of 
two and four singularities, respectively.
Furthermore, we will show the behavior of the energies, the singularities 
and the velocity fields for each of these simulations.  
Let us begin with the case of two singularities.

\subsubsection{Two singularities} This experiment was  originally proposed 
for a nematic liquid crystal without stretching in \cite{Liu2000} and, later extended for system \eqref{LC} in \cite{Lin2007}. In these works, Dirichlet boundary conditions for the director field were considered. In our case, 
we consider Neumann boundary conditions as in the results presented in \cite{Becker2008} and \cite{Cabrales2015}. The initial conditions of the problem are
 \[
\u_0=\boldsymbol{0},\quad
\dd_0=\frac{\tilde\dd}{\sqrt{|\tilde\dd|^2+\varepsilon^2}},
\mbox{ where } \tilde\dd=(x^2+y^2-0.25,y).
\]
In Figure \ref{dinamica2sing}, we present snapshots of the director and 
velocity fields displayed at  times $t=0.0, 0.1, 0.2,0.3$. Initially, 
the two singularities are transported to the origin by the velocity field, 
which, at the beginning of the experiment, forms four vortices being 
transformed only into two at the end. 
 %This vortices are symmetric with respect to $y=0$.
The numerical results show that the behavior of the director field is 
analogous qualitatively to the corresponding director field for the model 
without stretching reported in \cite{Becker2008} and \cite{Cabrales2015}. 
On the other hand, quantitatively, we observe that the size of the time interval, where the dynamics takes place, is very different. In our case, the size of this interval is about three times smaller than in the case without stretching, causing the annihilation time to be smaller than that in \cite{Cabrales2015}. Concerning the behavior of the velocity field, we find that both the magnitude and the dynamics of the velocity  are quite different. The maximum values of the kinetic energy in  \cite{Cabrales2015} and now differ about two times each other, being bigger now with stretching. We start the computation with the formation of four vortices, and end up with two symmetrical vortices with respect to $y=0$, whilst, in \cite{Becker2008} and \cite{Cabrales2015}, the four symmetrical vortices with respect to the origin keep so till the dynamics vanishes.

In Figure \ref{EnCinetica2Sing}, we present the evolution of kinetic, elastic, 
and  penalization energies, as well as the total energy. As predicted by 
inequality \eqref{induction}, the total  energy decreases with time. 
Also, the kinetic energy reaches its maximum level at  the annihilation 
time that in this case is at $t\thickapprox 0.242$. After this time, 
the system evolves to a steady state solution and all the energies 
decay quickly. Qualitatively, the graphics of the energies are similar 
to those reported in \cite{Becker2008} and \cite{Cabrales2015} 
for the model without stretching.
\begin{figure}[h]
\centering
\subfigure[$\|\dd\|_\infty=0.9977852.$]{
\includegraphics[scale=0.2]{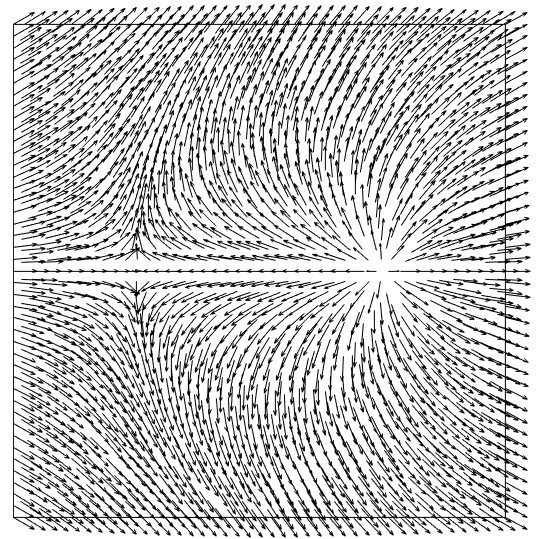}} 
\subfigure[$\|\dd\|_\infty=0.9971739.$]{
\includegraphics[scale=0.2]{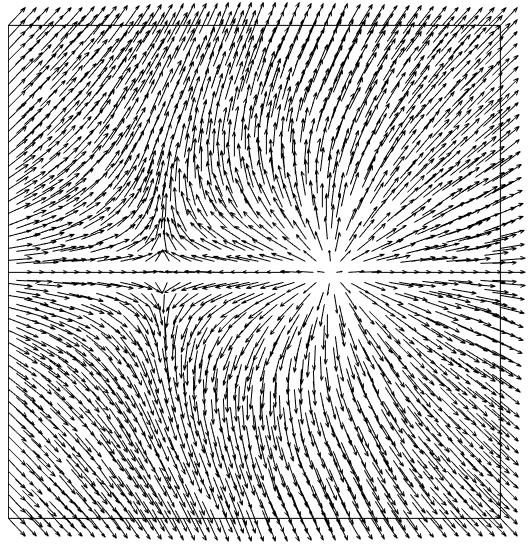}} 
\subfigure[$\|\dd\|_\infty=0.9973484.$]{
\includegraphics[scale=0.2]{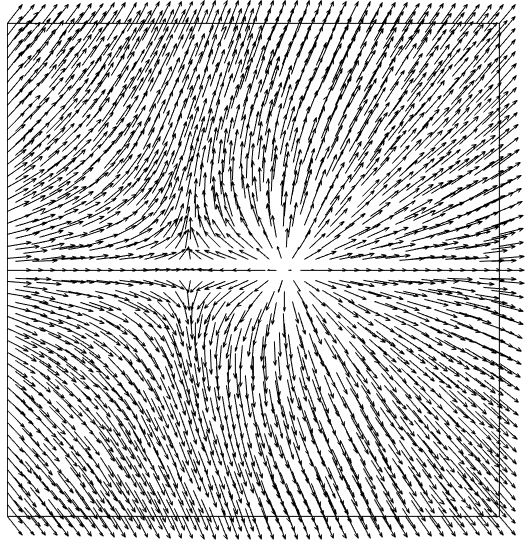}} 
\subfigure[$\|\dd\|_\infty=0.9974461.$]{
\includegraphics[scale=0.2]{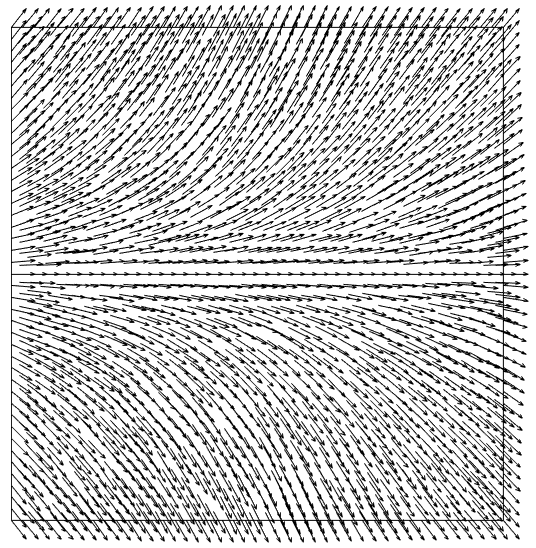}} 
\subfigure[$\|\u\|_\infty=0.0.$]{
\includegraphics[scale=0.21]{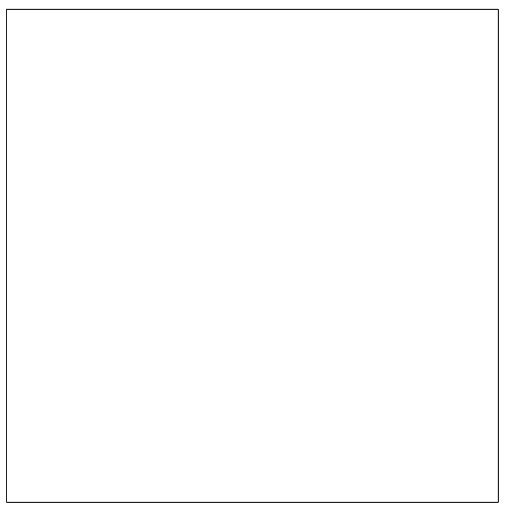}} 
\subfigure[$\|\u\|_\infty=0.289702.$]{
\includegraphics[scale=0.21]{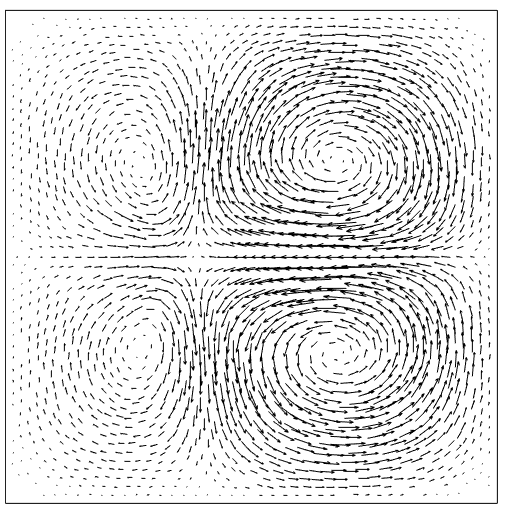}} 
\subfigure[$\|\u\|_\infty=0.4079601.$]{
\includegraphics[scale=0.21]{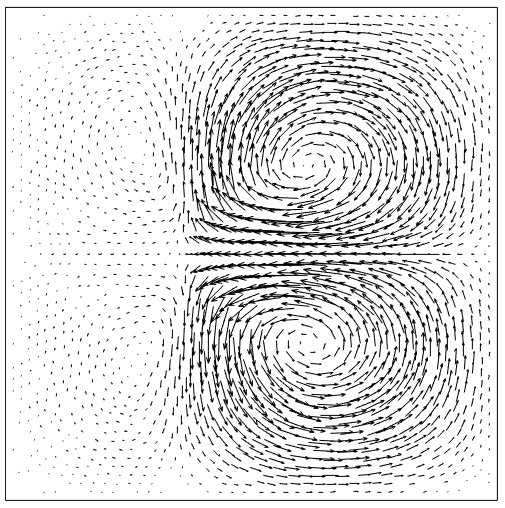}} 
\subfigure[$\|\u\|_\infty=0.2954702.$]{
\includegraphics[scale=0.21]{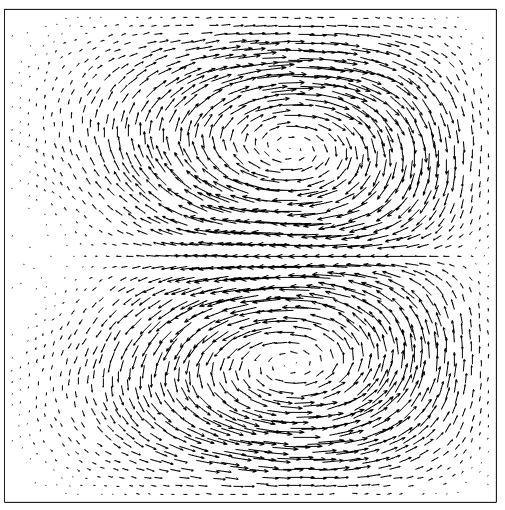}} 
\caption{Evolution of the director field (a)-(d) and the velocity field (e)-(h)
for the annihilation of two singularities
at times $t=0.0, 0.1, 0.2,0.3$. Here $\beta=-1$ and $H_F=0$.}
\label{dinamica2sing}
\end{figure}
\begin{figure}[h]
\centering
\subfigure[]{
\includegraphics[scale=0.36]{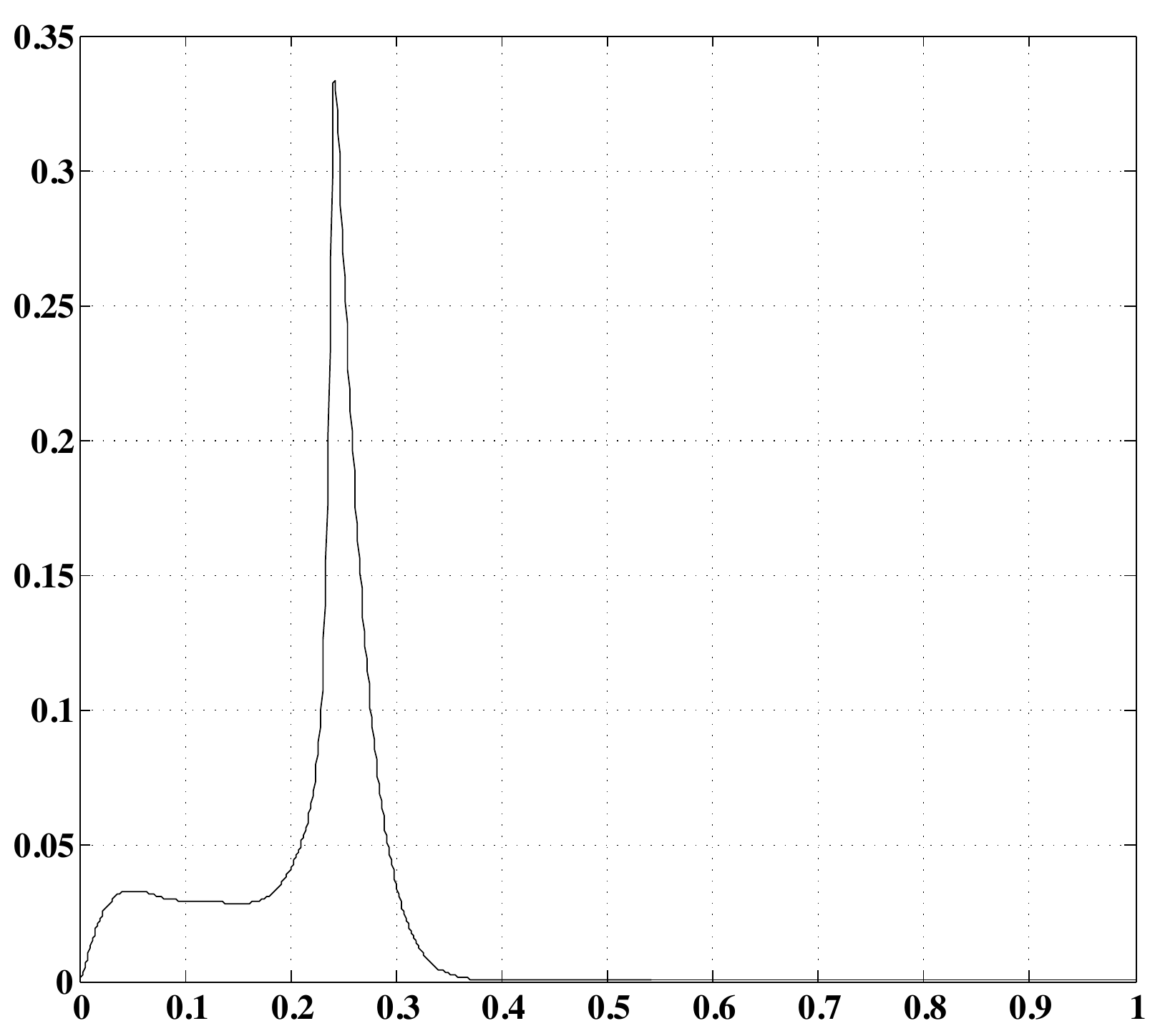}}\hspace{1cm}
\subfigure[]{
\includegraphics[scale=0.52]{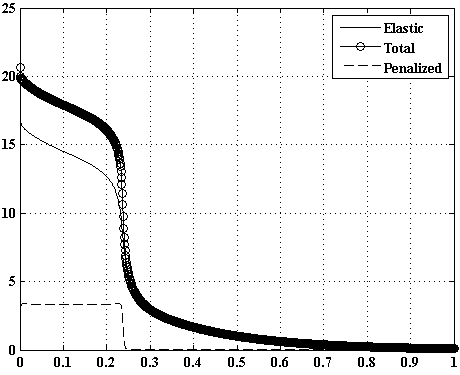}}%.pdf}}
\caption{Evolution in time of the energies for the experiment of two singularities.
Kinetic energy (left) and total, elastic, and penalization energies (right). 
Here $\beta=-1$ and $M=0$.}\label{EnCinetica2Sing}
\end{figure}

\subsubsection{Four singularities}

In this case, the initial director field has two singularities on the 
$x-$axis located at the points $(\pm 0.5,0)$ and two more singularities on 
the $y-$axis, located at the points  $(0,\pm0.25)$. More precisely,  
we consider the initial conditions
 \[
\u_0=\boldsymbol{0},\quad
\dd_0=\frac{\tilde\dd}{\sqrt{|\tilde\dd|^2+\varepsilon^2}},\mbox{ where } 
\tilde\dd=\left(\frac{x^2}{0.5^2}+\frac{y^2}{0.25^2}-1 ,- x y \right).
\]
In Figure \ref{dinamica4sing}, we present snapshots of the director 
and velocity fields displayed at times $t=0.02, 0.06, 0.08$ and $0.12$.
 We observe that the dynamics of the four singularities is faster than in 
 the two-singularity case. The two singularities located on the $x-$axis 
 begin to move toward each other, while those located on the $y-$axis 
 remains without moving till all the singularities are positioned at 
 the same distance from the origin. Then they move uniformly to the 
 origin and simultaneously disappear at the origin. This behavior is due 
 to the dynamics of the velocity field in the $y-$axis, preventing the 
 singularities located there from moving to the origin. 

As in the case of two singularities, we note that the annihilation time 
is smaller than in\cite{Cabrales2015}. Now, this time is at $t\thickapprox 0.071$ 
which is the half of the time reported in \cite{Cabrales2015}. 
This behavior is due to the increase of the magnitude of the velocity field.
 Additionally, at the beginning of the experiment, we note the formation 
 of four big vortices in each quadrant of the $xy$ plane and other 
 four more close to the $y-$axis. This last four vortices disappear 
 as time moves on. When the singularities annihilate, we note only 
 three vortices in the velocity field. 

\begin{figure}[H]
\centering
\subfigure[$\|\dd\|_\infty=0.9999915.$]{
\includegraphics[scale=0.18]{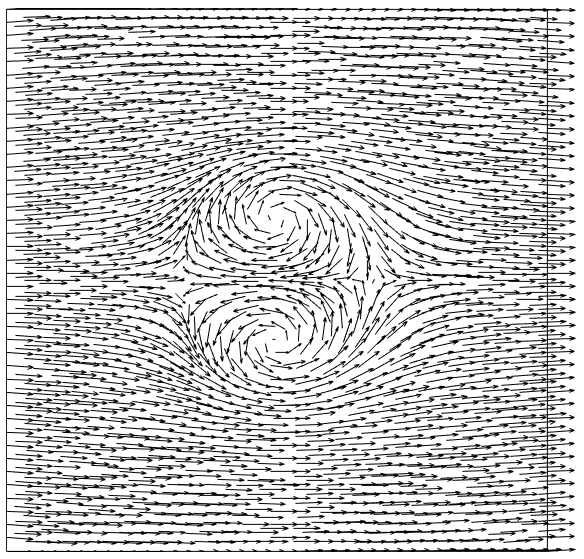}} 
\subfigure[$\|\dd\|_\infty=0.9999619$]{
\includegraphics[scale=0.18]{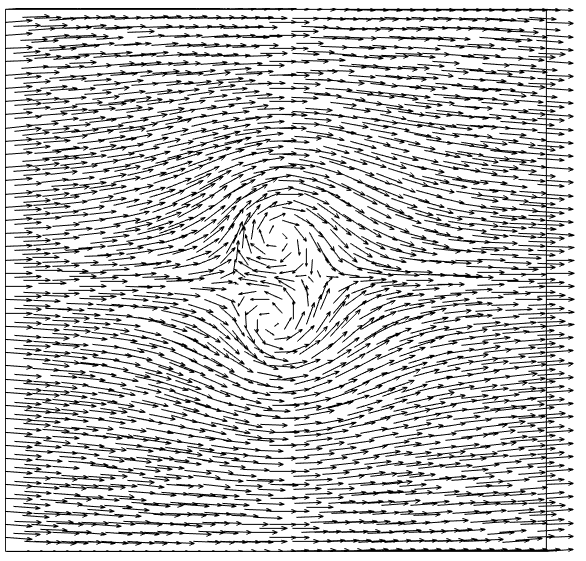}} 
\subfigure[$\|\dd\|_\infty=0.9999404.$]{
\includegraphics[scale=0.18]{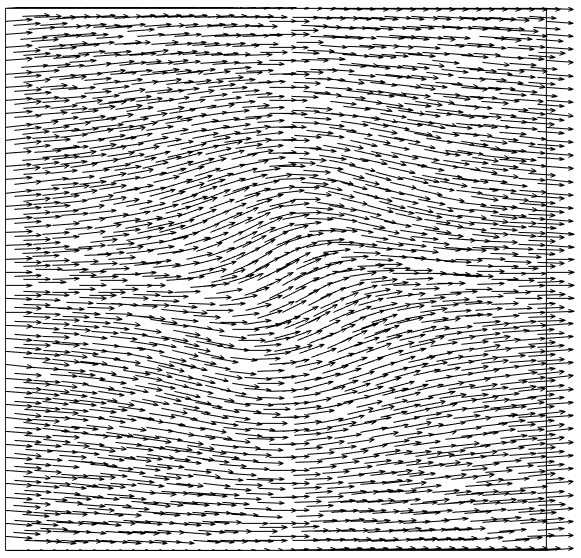}} 
\subfigure[$\|\dd\|_\infty=0.9999247.$]{
\includegraphics[scale=0.18]{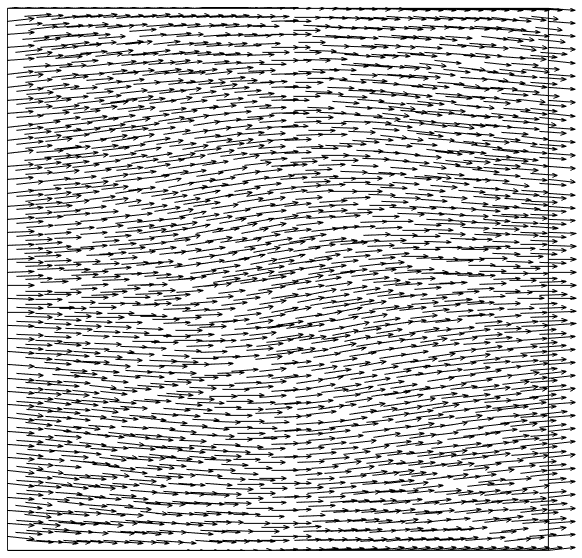}}
\subfigure[$\|\u\|_\infty=1.21298.$]{
\includegraphics[scale=0.18]{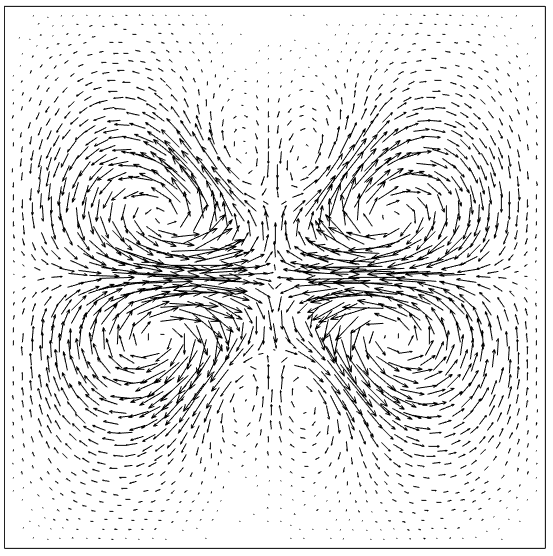}}
\hspace{0.1cm}
\subfigure[$\|\u\|_\infty=1.765187.$]{
\includegraphics[scale=0.18]{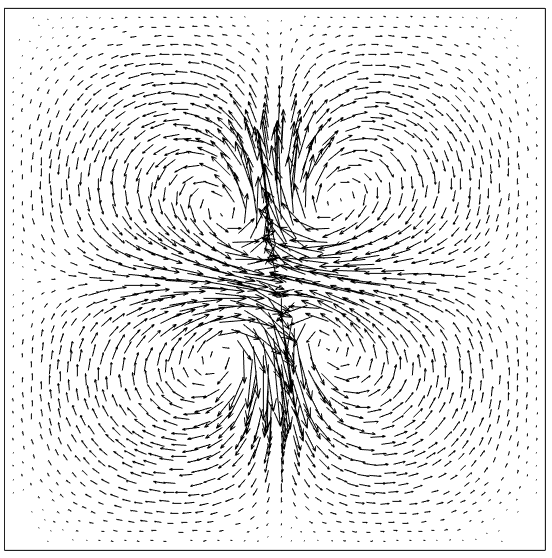}} 
\hspace{0.1cm}
\subfigure[$\|\u\|_\infty=1.299686.$]{
\includegraphics[scale=0.18]{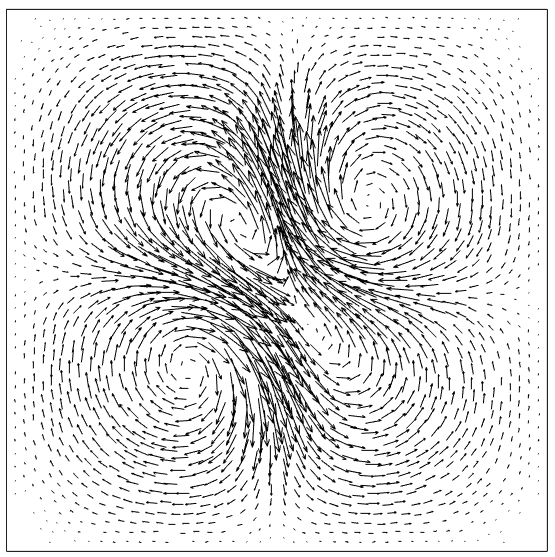}} 
\hspace{0.1cm}
\subfigure[$\|\u\|_\infty=0.2250148.$]{
\includegraphics[scale=0.18]{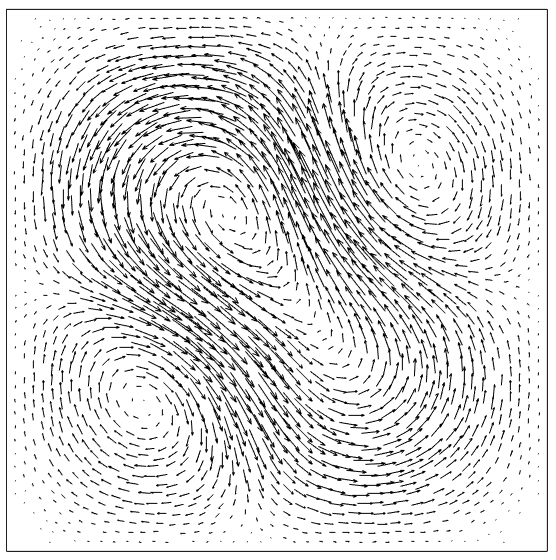}} 
\caption{Evolution of the director field (a)-(d) and the velocity field (e)-(h)
for the annihilation of four singularities at times 
$t=0.02, 0.06, 0.08,0.12$. Here $\beta=-1$ and $H_F=0$.}
\label{dinamica4sing}
\end{figure}

The graphs of the energies are shown in Figure \ref{EnCinetica4Sing}. In particular, we highlight that the maximum value of the kinetic energy is attained at about the annihilation time. On the other hand, as described in \cite{Cabrales2015}, such a maximum is reached at the very beginning of the simulation far away from the annihilation. The evolution of  the other energies behaves similarly. The total and elastic energies decrease with time, whilst the penalized one presents an increase at the beginning to adopt a constant behavior and, finally, 
decreases with time after the annihilation of the singularities.
\begin{figure}[H]
\centering
\subfigure[]{
\includegraphics[scale=0.33]{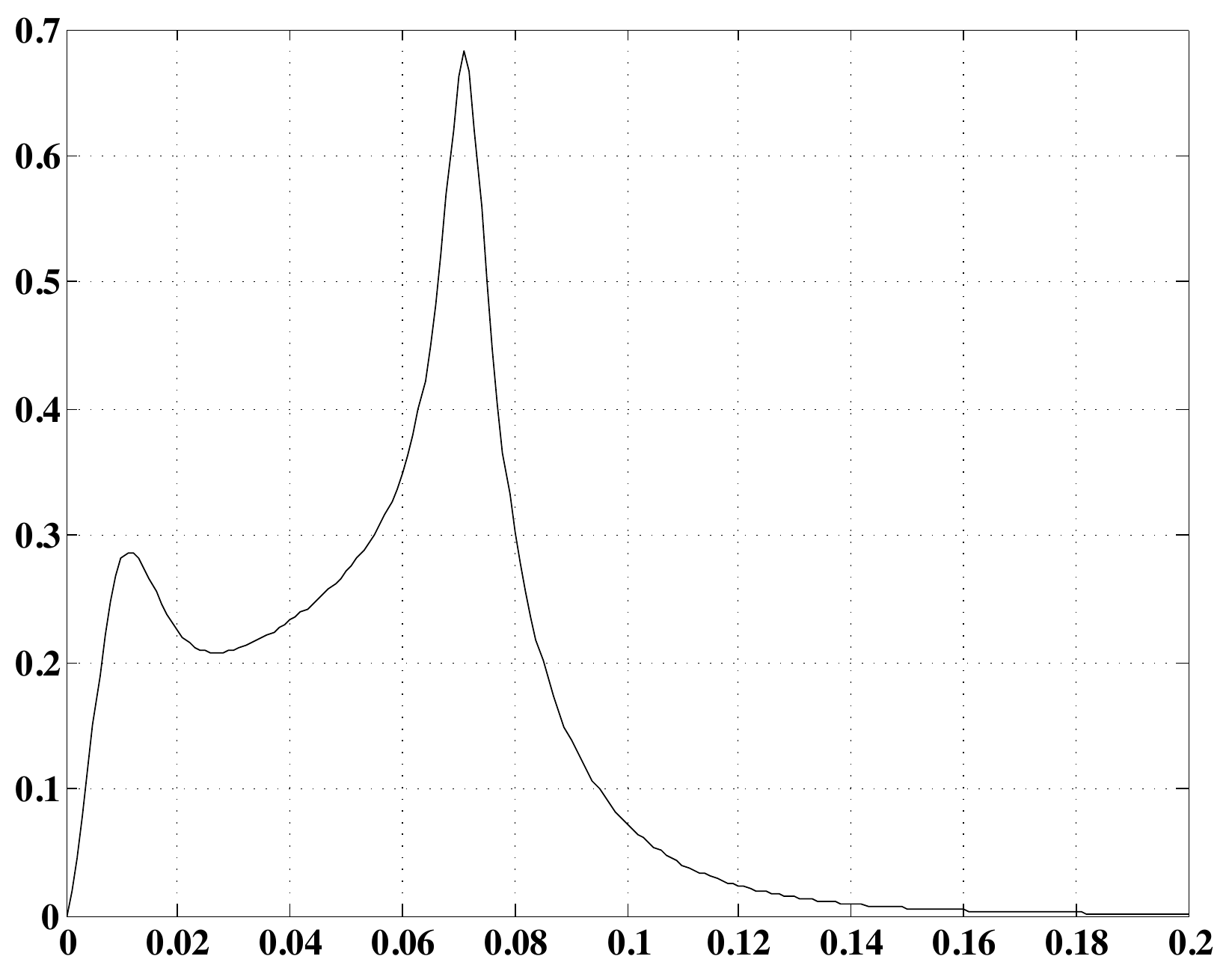}}\hspace{0.5cm}
\subfigure[]{
\includegraphics[scale=0.43]{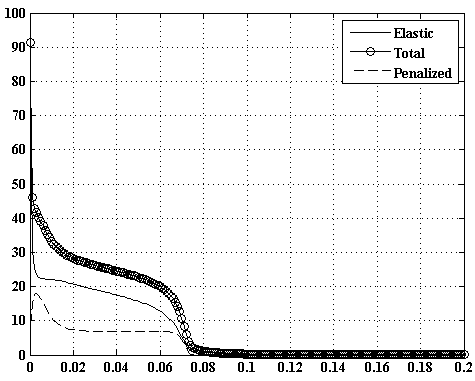}}%pdf}}
\caption{Evolution in time of the energies for the experiment of four singularities.
Kinetic energy (left) and total, elastic, and penalization energies (right).}
\label{EnCinetica4Sing}
\end{figure}

\subsection{Stability dependence on $H_F$}

Next, we carry out a sensitivity study of scheme 
\eqref{scheme3eq1}-\eqref{scheme3eq2}. More precisely, we give a detailed 
look to the relation among the stabilization constant $H_F$ defined in \eqref{constanteHf}, 
the geometrical parameter $\beta$ and the penalization parameter $\varepsilon$. 
The results were generated by considering the problems of annihilation of singularities described previously; see Tables  \ref{tabla1}, \ref{tabla2}, \ref{tabla1four} and \ref{tabla2four}.
\begin{table}[H]
\centering
\caption{Stability dependence of scheme \eqref{scheme3eq1}-\eqref{scheme3eq2} 
on the parameters $M$ and $\beta$ for the annihilation of two
singularities.
% with two and four singularities.
In this case $\varepsilon=0.05, k=0.001$ and $h=0.0913931$.
$T_{\max}$ is the time where the kinetic energy reaches its maximum value.}
\begin{tabular}{r|lllll|l}\hline
%\multicolumn{7}{c}{\textbf{Two singularities}}\\\hline
\textbf{$\beta \backslash M$} & 0 & 0.5  & 1 & 1.5 & 2 &\\\hline
           & \ding{51} & \ding{51}    & \ding{51}  & \ding{51}  &\ding{51}  & Stab.\\ 
0         & 0.293 &  0.412  & 0.486 & 0.542  & 0.628  & $T_A$\\ 
& 3.12172$\times 10^{-34}$&  1.084$\times 10^{-34}$  & 2.5296$\times 10^{-34}$ &  
1.1265$\times 10^{-34}$ & 9.61286$\times 10^{-34}$ & $E_{kin}$\\ \hline
%%%% segunda fila 
    & \ding{51} & \ding{51}    & \ding{51}  & \ding{51}  &\ding{51}  & Stab.\\ 
$-$0.2 & 0.302    & 0.426 &  0.501 &  0.571  &  0.64  &   $T_A$\\ 
 &0.0321312& 0.0214626 & 0.0176872 &  0.0150644  &  0.0131091 & $E_{kin}$\\ \hline
%%%% 3ra fila  
    & \ding{51} & \ding{51}    & \ding{51}  & \ding{51}  &\ding{51}  & Stab.\\ 
$-$0.5  & 0.285 & 0.41 & 0.485 & 0.555& 0.624 & $T_A$\\ 
  &  0.156098 &   0.0991694 & 0.0802572 & 0.0674879 & 0.0581182&  $E_{kin}$\\ \hline
 %%%%4ta fila
 & \ding{51} & \ding{51}    & \ding{51}  & \ding{51}  &\ding{51}  & Stab.\\ 
$-$0.8 & 0.259& 0.385  & 0.461 & 0.532 &  0.601& $T_A$\\ 
       &0.27635 & 0.16469  &  0.130359 & 0.108019 & 0.0918922 &$E_{kin}$\\ \hline
        %%%%5ta fila
  & \ding{51} & \ding{51}    & \ding{51}  & \ding{51}  &\ding{51}  & Stab.\\ 
$-$1.0 & 0.242 &   0.369 & 0.445 &   0.517 & 0.586 &$T_A$\\ 
   &0.332162 &  0.189956&0.148277& 0.121708 &  0.102846&$E_{kin}$\\ \hline
\end{tabular}\label{tabla1}
\end{table}

\subsubsection{Study of $M$ vs. $\beta$}
We are concerned with the dependence of the stability on the parameters $H_F$ and $\beta$. The results are presented in Tables  \ref{tabla1} and \ref{tabla1four}. The former corresponds to the two-singularity case and the latter corresponds to the four-singularity case. We take $(\varepsilon, k,h)=(0.05, 0.001,0.0913931)$ and vary  
\[
(\beta,M)\in\{0, -0.2, -0.5, -0.8, -1\}\times \{0, 0.5, 1, 1.5, 2\}.
\] 
For each pair $(\beta,M)$ and both cases of singularities, Tables
\ref{tabla1} and \ref{tabla2} say us that the algorithm is stable and, 
that the geometrical parameter $\beta$ has a little or no influence in the 
stability of scheme \eqref{scheme3eq1}-\eqref{scheme3eq2}. An analysis 
similar to that in \cite{Guillen2013} leads us to think that the selected 
values for the parameters $(\varepsilon, k, h)$ are such that they must 
satisfy a certain relation among them. Thus scheme 
\eqref{scheme3eq1}-\eqref{scheme3eq2} is stable for $H_F=0$ and, 
consequently, for $H_F>0$. 
In Table \ref{tabla1}  
we also provide the maximum value of the kinetic energy $E_{kin}$ and 
the corresponding time $T_A$ (called annihilation time) at which is reached. 
If we fix $\beta$ and move the values of $H_F$ (taking $M$ from 0 to 2), 
we observe that the value of the kinetic energy decreases whilst the value 
of the annihilation time $T_A$ increases. On the contrary, if we fix $H_F$ 
and move the values of $\beta$ from -0.2 to -1, we observe that the value 
of the kinetic energy increases whilst the value of the annihilation 
time $T_A$ decreases. The only exceptional case that does not follow 
this pattern is for $\beta=0$, where the kinetic energy is almost zero and, 
however, we have annihilation for the cases of two and four singularities.
%%% resultados para las 4
\begin{table}[H]
\centering
\caption{Stability dependence of scheme \eqref{scheme3eq1}-\eqref{scheme3eq2} on the
parameters $M$ and $\beta$ for the annihilation of four
singularities.% with two and four singularities. 
In this case $\varepsilon=0.05, k=0.001$ and $h=0.0913931$.
$T_{\max}$ is the time where the kinetic energy reaches its maximum value.}
\begin{tabular}{r|lllll|l}\hline
%\multicolumn{7}{c}{\textbf{Four singularities}}\\\hline
$\beta \backslash M$ & 0 & 0.5 &    1.0        &  1.5         & 2.0\\\hline
 & \ding{51} & \ding{51}   & \ding{51}  & \ding{51}  &\ding{51}  & Stab.\\ 
0 & 0.065 & 0.097 & 0.115  & 0.132 & 0.148 &   $T_A$\\ 
 & 1.74106$\times 10^{-34}$  & 5.80466$\times 10^{-34}$ & 3.52701$\times 10^{-34}$  &
   9.66597$\times 10^{-34}$ & 5.92986$\times 10^{-34}$ &   $E_{kin}$\\ \hline
%%% 2da fila
 & \ding{51} & \ding{51}   & \ding{51}  & \ding{51}  &\ding{51}  & Stab.\\ 
$-$0.2 & 0.07 & 0.099&  0.117 & 0.134 & 0.151 &  $T_A$\\ 
 & 0.0209245 & 0.0119331 & 0.00905435 & 0.00720009 & 0.00588037 &  $E_{kin}$\\ \hline
%%% 3ra fila
 & \ding{51} & \ding{51}   & \ding{51}  & \ding{51}  &\ding{51}  & Stab.\\ 
$-$0.5 & 0.073 & 0.102 & 0.12 & 0.137 & 0.153 &  $T_A$\\ 
 & 0.140803 & 0.0811353 & 0.061811 & 0.0493141 & 0.0404074 &  $E_{kin}$\\ \hline
%%% 4ta fila
 & \ding{51} & \ding{51}   & \ding{51}  & \ding{51}  &\ding{51}  & Stab.\\ 
$-$0.8& 0.073 & 0.103 & 0.12 & 0.137 & 0.153 &  $T_A$\\ 
     & 0.357074 & 0.204117 & 0.155465 & 0.124123 & 0.101823 &  $E_{kin}$\\ \hline
%%% 5ta fila
 & \ding{51} & \ding{51}   & \ding{51}  & \ding{51}  &\ding{51}  & Stab.\\ 
$-$1.0 & 0.071 & 0.101 & 0.119 & 0.135 & 0.152 &  $T_A$\\ 
&  0.526868&  0.299899 & 0.227447 & 0.181156 & 0.148931 &   $E_{kin}$\\ \hline
\end{tabular}\label{tabla1four}
\end{table}

\begin{table}[H]
\centering
\caption{Stability dependence of scheme \eqref{scheme3eq1}-\eqref{scheme3eq2} on the
parameters $M$ and $\varepsilon$ for the annihilation phenomenon. 
% with two singularities.
In this case $\beta=-1, k=0.001$ and $h=0.0913931$.
$T_{\max}$ is the time where the kinetic energy reaches its maximum value.}
%\hspace*{-5em}
\begin{tabular}{r|lllll|l}\hline
%\multicolumn{6}{c}{\textbf{Four singularities}}\\\hline                                                     
\textbf{$\varepsilon \backslash M$} & 0 & 0.5  & 1 & 1.5 & 2 &\\\hline
  \multirow{ 3}{*}{0.1} & \ding{51} & \ding{51} & \ding{51} & 
  \ding{51} & \ding{51} & Stab.\\ 
     & 0.041     & 0.045  &   0.047 & 0.048 & 0.05 &$T_{\max}$\\ 
  & 0.51046 & 0.441406 & 0.406151 & 0.376596 & 0.350638 &$E_{kin}$\\ \hline
 \multirow{ 3}{*}{0.05} & \ding{51} & \ding{51} & \ding{51} & \ding{51} 
 & \ding{51} &  Stab.\\ 
 & 0.071 & 0.1 & 0.118 &  0.134 & 0.151 &$T_{\max}$\\ 
  & 0.682407 & 0.419277 & 0.332545  & 0.273872 & 0.231861 &$E_{kin}$\\ \hline
 \multirow{ 3}{*}{0.01}  &  \ding{55} & \ding{51} & \ding{51} & \ding{51} 
 & \ding{51} &  Stab.\\ 
 & $--$ & \textbf{No annihil.} & \textbf{No annihil.} & \textbf{No annihil.} 
 & \textbf{No annihil.} &$T_{\max}$\\ 
 & $--$ &  0.258405 & 0.169371 &  0.123507 & 0.096421 &$E_{kin}$\\ \hline
 \multirow{ 3}{*}{0.001} & \ding{55} & \ding{51} & \ding{51} & \ding{51} 
 & \ding{51} &  Stab.\\ 
 & $--$ &\textbf{No annihil.} & \textbf{No annihil.} 
 & \textbf{No annihil.} & \textbf{No annihil.} &$T_{\max}$\\ 
&  $--$ & 0.167737 & 0.0974185 & 0.06874 & 0.0524003 &$E_{kin}$\\ \hline
\end{tabular}\label{tabla2four}
\end{table}
\subsubsection{Study of $M$ vs. $\varepsilon$}

We are now interested in the dependence of the stability on the parameters $H_F$ and $\varepsilon$. The results are presented in Tables  \ref{tabla2} and \ref{tabla2four}. As before, the former is for the two-singularity case, while the latter is for  the four-singularity case. We fix $(\beta, k,h)=(-1,0.001,0.0913931)$ and move
\[
(\varepsilon,M)\in
\{0.1, 0.05, 0.01, 0.001\}\times \{0, 0.5, 1, 1.5, 2\}. 
\]
The results are similar to those presented in \cite[Table 3]{Cabrales2015} 
for the two-singularity case. In fact, for both cases of singularities, 
scheme \eqref{scheme3eq1}-\eqref{scheme3eq2} is unconditionally stable 
for $H_F$ such that $M \ge 0.5$ and conditionally stable for $H_F = 0$, 
where strong spurious oscillations appear  for $\varepsilon = 0.01$ and $0.001$.
 Moreover, for $\varepsilon= 0.1$ and $0.05$, the annihilation time
  becomes smaller and smaller as $H_F$ decreases to 0, and the maximum 
  of the kinetic energy decreases as $H_F$ becomes bigger and bigger, 
  but the qualitative behavior remains the same. This situation changes 
  drastically as $\varepsilon$ takes the values $0.01$ and $0.001$ where 
  there is no longer annihilation.  Figure \ref{EnCin2y4Sing} shows that 
  the kinetic energy decreases for two- and four-singularity cases, being 
  larger in the two-singularity case.  As was pointed in \cite{Cabrales2015}, 
  a possible explanation of this behavior might be that the velocity field 
  produced via the elastic tensor is not enough to move the singularity 
  points through the convective term in the director equation. In particular, 
  for $\varepsilon = 0.01$ and $0.001$, the kinetic energy  decays practically 
  to zero from the beginning. In light of the above, one might think that 
  if the kinetic energy associated to a velocity field was large enough 
  to move the singularities, then they would move each other.
 \begin{table}[H]
\centering
\caption{Stability dependence of scheme \eqref{scheme3eq1}-\eqref{scheme3eq2} on the
parameters $M$ and $\varepsilon$ for the annihilation phenomenon. 
% with two singularities.
In this case $\beta=-1, k=0.001$ and $h=0.0913931$.
$T_{\max}$ is the time where the kinetic energy reaches its maximum value.}
%\hspace*{-5em}
\begin{tabular}{r|lllll|l}\hline
%$\frac{\varepsilon}{M}$   0    
%\multicolumn{6}{c}{\textbf{Two singularities}}\\\hline
%, ${\color{blue} \beta=-1.0}$}}\\\hline
$\varepsilon \backslash M$ & 0 & 0.5 &    1.0        &  1.5         & 2.0\\\hline
%%1ra fila
 \multirow{ 3}{*}{0.1}&  \ding{51}&  \ding{51}&  \ding{51}&  \ding{51}&  
 \ding{51}&Stab.\\ 
   & 0.168 & 0.188 & 0.2 & 0.211 & 0.222  &$T_{\max}$\\ 
   & 0.2392326 & 0.2011553 & 0.1828 & 0.1679  & 0.1551 &$E_{kin}$\\ \hline
%%2da fila
 \multirow{ 3}{*}{0.05}&\ding{51}&  \ding{51}&  \ding{51}&  \ding{51}& 
  \ding{51}&Stab.\\ 
 & 0.242 & 0.369  & 0.445 &  0.516 & 0.585 &$T_{\max}$\\ 
 & 0.3335189 & 0.1909416 & 0.1490 & 0.1222 & 0.1032  &$E_{kin}$\\ \hline
%%3ra fila
 \multirow{ 3}{*}{0.01} &  \ding{55}&  \ding{51}  & \ding{51}&  
 \ding{51}&  \ding{51}& Stab.\\ 
 & $--$  & \textbf{No annihil.} & \textbf{No annihil.}  & \textbf{No annihil.} &  
 \textbf{No annihil.} &$T_{\max}$\\ 
 & $--$  & 0.007549 & 0.0032 & 0.0018 & 0.0014  &$E_{kin}$\\ \hline
%%4ta fila
 \multirow{ 3}{*}{0.001}&\ding{55}&  \ding{51}&  \ding{51}&  
 \ding{51}&  \ding{51}&Stab.\\ 
 & $--$ &  \textbf{No annihil.} &  \textbf{No annihil.} &    \textbf{No annihil.} & 
  \textbf{No annihil.} &$T_{\max}$\\ 
& $--$ & 0.105065  &   0.004  & 0.0022 &    0.0016  & $E_{kin}$\\ \hline
\end{tabular}\label{tabla2}
\end{table}

\begin{figure}[H]
\centering
\subfigure[]{
\includegraphics[scale=0.41]{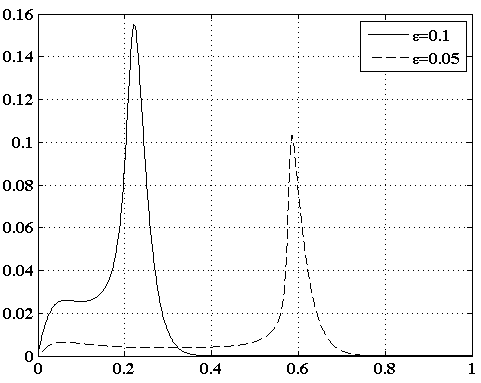}}
\hspace{0.5cm}
\subfigure[]{
\includegraphics[scale=0.41]{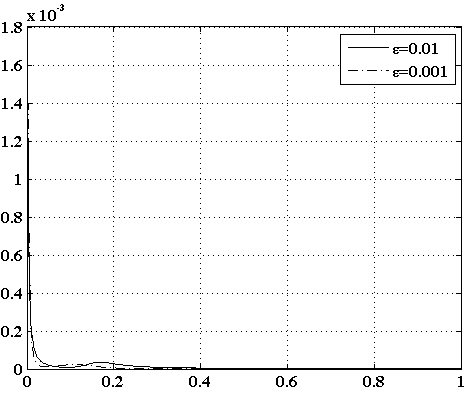}}\\
\subfigure[]{
\includegraphics[scale=0.41]{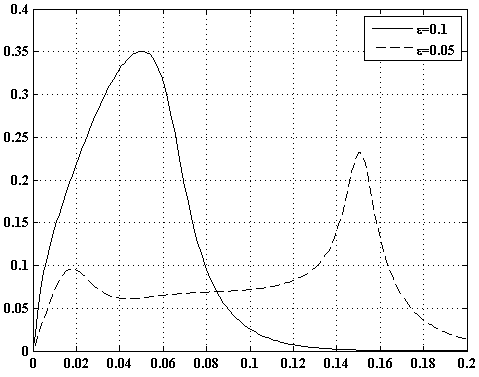}}\hspace{0.5cm}
\subfigure[]{
\includegraphics[scale=0.41]{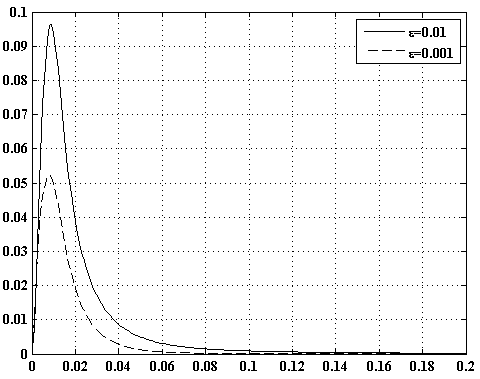}}
\caption{Evolution in time of the kinetic energy for 
$\varepsilon = 0.1, 0.05, 0.01, 0.001$. Top figures
are for the two singularities experiment, meanwhile,
bottom figures are for the four singularities experiment. In this
case, we consider $M=2.0$. }\label{EnCin2y4Sing}
\end{figure}

 \begin{figure}[H]
\centering
\subfigure[$\|\dd\|_\infty=1.11139$]{
\includegraphics[scale=0.33]{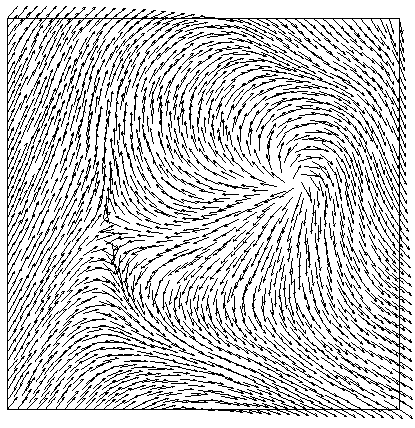}}
%\hspace{0.5cm}
\subfigure[$\|\dd\|_\infty=1.01984$]{
\includegraphics[scale=0.33]{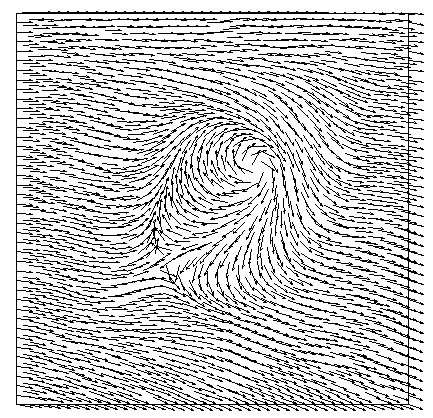}}
\subfigure[$\|\dd\|_\infty=1.06926$]{
\includegraphics[scale=0.33]{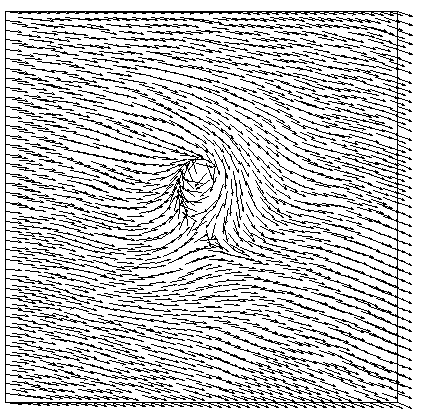}}
%\hspace{0.5cm}
\subfigure[$\|\dd\|_\infty=0.992742$]{
\includegraphics[scale=0.33]{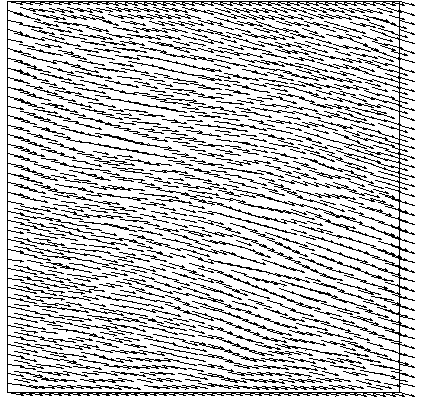}}
\subfigure[$\|\dd\|_\infty=1.04313$]{
\includegraphics[scale=0.33]{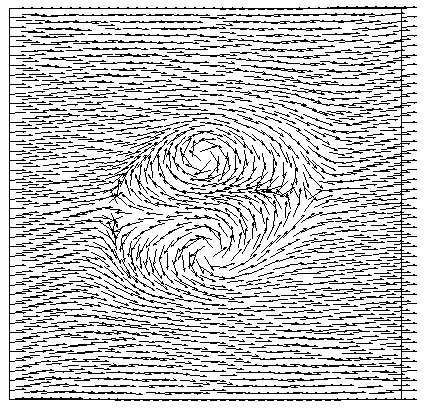}}
%\hspace{0.5cm}
\subfigure[$\|\dd\|_\infty=1.06166$]{
\includegraphics[scale=0.33]{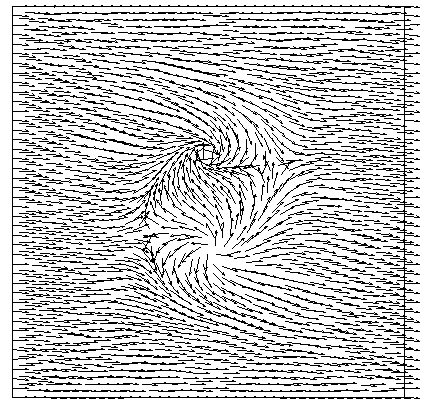}}
\subfigure[$\|\dd\|_\infty=1.03869$]{
\includegraphics[scale=0.33]{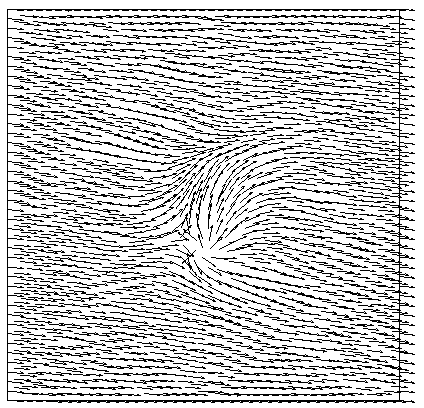}}
%\hspace{0.5cm}
\subfigure[$\|\dd\|_\infty=1$]{
\includegraphics[scale=0.33]{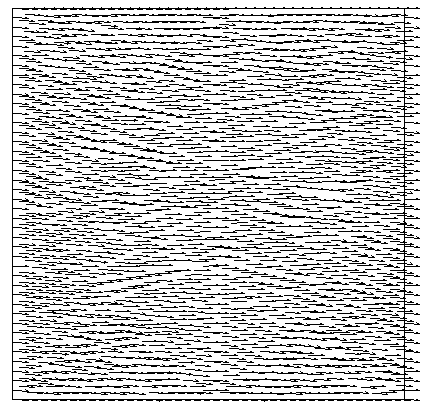}}
\caption{Evolution of the director field for a rotating flow. 
Top figures: two singularities experiment at times 
$t = 0.2, 0.5, 2.0, 3$. Bottom figures: four singularities
experiment at times  $t = 0.1, 0.6, 1.3, 3$.}\label{EnCin2y4Singvelnonula}
\end{figure}

\end{document}